\documentclass[10pt]{article}

\usepackage{amsmath, amsthm, amssymb, enumerate}
\usepackage{color}
\usepackage{datetime}
\usepackage{fancyhdr}

\chardef\coloryes=0 %%%out
\chardef\isitdraft=0 %%%out
\ifnum\isitdraft=1 %%%out
   \def\eqref#1{({\ref{#1}})}                %saves writing paranthesis%%%out

\usepackage[color,notcite]{showkeys}%%%out

\usepackage[notcite]{showkeys}%%%out

\definecolor{labelkey}{gray}{.3}%%%out

\definecolor{refkey}{rgb}{.3,0.3,0.3}%%%out

\else%%%out

  \def\startnewsection#1#2{\section{#1}\label{#2}\setcounter{equation}{0}}   %\starts a new section
  \def\nnewpage{} %\nnewpage does nothing
\fi%%%out

\begin{document}
%\newcommand{\llabel}{\label}             %\llabel is just a synonim for \label
%\newcommand{\rref}{\ref}                 %\rref is  just a synonim for \ref
%\newcommand{\ccite}{\cite}               %\ccite is just a synonim for
                                         %when ref. to equations
\def\ques{{\cor \underline{??????}\cob}}
\def\nto#1{{\coC \footnote{\em \coC #1}}}
\def\fractext#1#2{{#1}/{#2}}
\def\fracsm#1#2{{\textstyle{\frac{#1}{#2}}}}   %smaller version of frac
\def\nnonumber{}

%\newcommand{\bv}{{u}}

%if \isitdraft=0 or \coloryes=0%%%out
\def\cor{{}}%%%out
\def\cog{{}}%%%out
\def\cob{{}}%%%out
\def\coe{{}}%%%out
\def\coA{{}}%%%out
\def\coB{{}}%%%out
\def\coC{{}}%%%out
\def\coD{{}}%%%out
\def\coE{{}}%%%out
\def\coF{{}}%%%out
%\def\MR#1{}%%%out

%%3
%\ifnum\isitdraft=1%%%out
\ifnum\coloryes=1%%%out

  \definecolor{coloraaaa}{rgb}{0.1,0.2,0.8}%%%out
  \definecolor{colorbbbb}{rgb}{0.1,0.7,0.1}%%%out
  \definecolor{colorcccc}{rgb}{0.8,0.3,0.9}%%%out
  \definecolor{colordddd}{rgb}{0.0,.5,0.0}%%%out
  \definecolor{coloreeee}{rgb}{0.8,0.3,0.9}%%%out
  \definecolor{colorffff}{rgb}{0.8,0.3,0.9}%%%out
  \definecolor{colorgggg}{rgb}{0.5,0.0,0.4}%%%out

 \def\cog{\color{colordddd}}%%%out
 \def\cob{\color{black}}%%%out
 \def\cor{\color{red}}%%%out
 \def\coe{\color{colorgggg}}%%%out

 \def\coA{\color{coloraaaa}}%%%out
 \def\coB{\color{colorbbbb}}%%%out
 \def\coC{\color{colorcccc}}%%%out
 \def\coD{\color{colordddd}}%%%out
 \def\coE{\color{coloreeee}}%%%out
 \def\coF{\color{colorffff}}%%%out
 \def\coG{\color{colorgggg}}%%%out

%%4
\fi%%%out
%\fi%%%out
\ifnum\isitdraft=1%%%out
   \chardef\coloryes=1 %%%out
   \baselineskip=17pt%%%out
   \input macros.tex%%%out
   \def\blackdot{{\color{red}{\hskip-.0truecm\rule[-1mm]{4mm}{4mm}\hskip.2truecm}}\hskip-.3truecm}%%%out
   \def\bdot{{\coC {\hskip-.0truecm\rule[-1mm]{4mm}{4mm}\hskip.2truecm}}\hskip-.3truecm}%%%out
   \def\purpledot{{\coA{\rule[0mm]{4mm}{4mm}}\cob}}%%%out
   \def\pdot{\purpledot}%%%out
\else%%%out  
   \baselineskip=15pt
   \def\blackdot{{\rule[-3mm]{8mm}{8mm}}}%%%out
   \def\purpledot{{\rule[-3mm]{8mm}{8mm}}}%%%out
   \def\pdot{}
%%5
\fi%%%out

\def\tdot{\fbox{\fbox{\bf\tiny I'm here; \today \ \currenttime}}}
\def\nts#1{{\hbox{\bf ~#1~}}} %nts=note to self
\def\nts#1{{\cor\hbox{\bf ~#1~}}} %nts=note to self%%%out
\def\ntsf#1{\footnote{\hbox{\bf ~#1~}}} %nts=note to self
\def\ntsf#1{\footnote{\cor\hbox{\bf ~#1~}}} %nts=note to self%%%out
\def\bigline#1{~\\\hskip2truecm~~~~{#1}{#1}{#1}{#1}{#1}{#1}{#1}{#1}{#1}{#1}{#1}{#1}{#1}{#1}{#1}{#1}{#1}{#1}{#1}{#1}{#1}\\}%%%out
\def\biglineb{\bigline{$\downarrow\,$ $\downarrow\,$}}%%%out
\def\biglinem{\bigline{---}}%%%out
\def\biglinee{\bigline{$\uparrow\,$ $\uparrow\,$}}%%%out

\def\tilde{\widetilde}

\newtheorem{Theorem}{Theorem}[section]
\newtheorem{Corollary}[Theorem]{Corollary}
\newtheorem{Proposition}[Theorem]{Proposition}
\newtheorem{Lemma}[Theorem]{Lemma}
\newtheorem{Remark}[Theorem]{Remark}
\newtheorem{assumption}[Theorem]{Assumption}
\newtheorem{definition}{Definition}[section]
\def\theequation{\thesection.\arabic{equation}}
\def\endproof{\hfill$\Box$\\}
\def\square{\hfill$\Box$\\}
\def\comma{ {\rm ,\qquad{}} }            %comma in a formula
\def\commaone{ {\rm ,\qquad{}} }         %second comma in a formula
\def\dist{\mathop{\rm dist}\nolimits}    %distance
\def\sgn{\mathop{\rm sgn\,}\nolimits}    %sgn
\def\Tr{\mathop{\rm Tr}\nolimits}    %trace
\def\div{\mathop{\rm div}\nolimits}    %divergence
\def\supp{\mathop{\rm supp}\nolimits}    %divergence
\def\divtwo{\mathop{{\rm div}_2\,}\nolimits}    %two dimensional divergence
\def\re{\mathop{\rm {\mathbb R}e}\nolimits}    %distance

%added by ie
\def\dbR{\mathbb R}
\def\dbL{\mathbb L}
\def\dbN{\mathbb N}
\def\dbP{\mathbb P}
\def\dbE{\mathbb E}
\def\dbF{\mathbb F}
\def\dbG{\mathbb G}
\def\dbQ{\mathbb Q}
\def\CC{{\tilde C}}
\def\cZ{{\cal Z}}
\def\cD{{\cal D}}

\def\indeq{\qquad{}}                     %indentation in formulas
\def\period{.}                           %period in a formula
\def\semicolon{\,;}                      %semicolon in a formula
%**end of header

\title{Existence of invariant measures for the stochastic damped  KdV equation}
\author{Ibrahim Ekren, Igor Kukavica, and Mohammed Ziane}
\maketitle

\date{}

\begin{center}
\end{center}

\medskip

\indent Departement fur Mathematik, ETH Zurich, Ramistrasse 101, CH-8092, Zurich\\
\indent email: ibrahim.ekren@math.ethz.ch\\
\indent Department of Mathematics, University of Southern California, Los Angeles, CA 90089\\
\indent e-mails: kukavica\char'100usc.edu, ziane\char'100usc.edu

\begin{abstract}
We address the long time behavior of solutions of the stochastic
Korteweg-de Vries equation
$    du + (\partial^3_x u +u\partial_x u +\lambda u)dt  = f dt+\Phi dW_t$
on ${\mathbb R}$ where $f$ is a deterministic force. 
We prove that the Feller property holds and establish the
existence
of an invariant measure. The tightness is established with the help of the
asymptotic compactness, which is carried out  using the
Aldous criterion.
\end{abstract}

\noindent\thanks{\em Mathematics Subject Classification\/}:
%35R35, %Free boundary problems
%35Q30, %Stokes and Navier-Stokes equations
%76D05  %Navier-Stokes equations 

\noindent\thanks{\em Keywords:\/}
Invariant measures, Stochastic KdV equation, White noise, Long time behavior, Asymptotic compactness, tightness, Feller property, Aldous criterion.
%Navier-Stokes equations, fluid-structure interaction, long time
%behavior, global solutions, damped wave equation

\startnewsection{Introduction}{sec1} 
In this paper, we investigate the long time behavior
of solutions of the stochastic damped KdV equation
  \begin{equation}
    du + (\partial^3_x u +u\partial_x u +\lambda u)dt
    =  f  dt +
    \Phi dW_t,
   \label{kdv-intro-A}
  \end{equation}
with a nonzero deterministic force, by establishing
the existence of an invariant measure. 
%A well-known procedure is the
%use of the Krylov-Bogoliubov theorem which relies on two important
%features of an SPDE; namely the Feller property of the semigroup and
%the tightness for integral time averages.

Invariant measures play a crucial role in understanding the long time
dynamics of solutions of stochastic partial differential equations \cite{CGV,DO,GKVZ,F1,F2}. In
particular, they were constructed for the stochastic Navier-Stokes system
\cite{F1}, the stochastic conservation laws \cite{DV}, the 
stochastic primitive equations \cite{GKVZ}, and for many
other equations and systems in mathematical physics.
However, as far as we know, the existence of an invariant measure for the
stochastic damped KdV equation is open, the difficulties
being the non-compactness of the domain and the 
asymptotic compactness of the semigroup.

%The existence of invariant
%measure for SPDEs has attracted interest of many mathematicians in
%several fields of applied mathematics. 

The existence and uniqueness of solutions
for the stochastic KdV equation has been established by 
de~Bouard and Debussche in \cite{DD1} (cf.~also \cite{Da,DD2,DD3,DD4,DP,O}).
However,
the existence of invariant measure, which by the Krylov-Bogoliubov
procedure requires
the Feller property and the tightness property of the time averages,
has not been established except when including
additional dissipative terms and in bounded domains \cite{Ku1,Ku2}.

There  are two main difficulties in carrying out the Krylov-Bogoliubov procedure for of the stochastic damped KdV equation. 
The first difficulty is related to establishing  the Feller
property, whose proof usually follows
from a~priori estimates on the solutions and the dominated convergence theorem. However, in the case of the KdV equation,
there are no a~priori estimates up to deterministic
times. In order to circumvent this difficulty, we use the results in \cite{DD1} and  establish a~priori
estimates up to some stopping times, which are then used to show the Feller property of
the transition semigroup.

The second and the main difficulty 
in establishing the existence of an invariant measure for \eqref{kdv-intro-A}
is the tightness of  the time averages. In fact all known approaches 
fail due to the lack of compactness and dissipation.
Thus, in order to obtain the tightness of averages, we are led to give a unconventional  proof. Indeed, we study the
equation on the whole domain which is unbounded  with non-compact Sobolev embeddings. To show the tightness, we first
use the existence
results in \cite{DD1} in order to establish uniform estimates on the solutions of the 
equation. These bounds give us tightness of measures on the space
$L^2_{\rm loc}(\dbR)$ of locally square integrable functions.  To pass
from tightness in $L^2_{\rm loc}(\dbR)$ to tightness in $L^2(\dbR)$, one
intuitively needs to show that there is no mass escaping to infinity,
which in this stochastic framework means the convergence of the
expectation of the square of the $L^2(\dbR)$ norm to the 
expectation of the square of the $L^2(\dbR)$ norm of the limiting
measure. We then use a result in \cite{Pr} on the convergence in
measure in Hilbert spaces, to obtain the tightness in $L^2(\dbR)$ and
$H^1(\dbR)$.

%For the this purpose we need to adapt the John Ball
%method of energy equtions to the probabilistic setting. 

%\nts{Ibrahim: summarize the steps of the proof - i.e., convergence of
%$l^2$ nornms and of I; then mention Althous criterion and Rosa's
%convergence technique}.

In the deterministic case, there is a vast literature on the 
well-posedness of solutions of the KdV equation. Starting with
the seminal  of Temam \cite{T} , who 
established the global existence of weak solutions in
$H^1$, then the
existence and uniqueness in  Sobolev spaces
was established by Kato (cf.~\cite{K}). The well-posedness theory was further
studied by Bona-Smith \cite{BS1,BS2}, 
Saut-Temam \cite{ST}, Bourgain \cite{Bou}, Kenig~et~al \cite{KPV1,KPV2}, 
Colliander et al \cite{CKSTT},
and by many other authors.
The long time behavior of the KdV
equations
was initially studied by Ghidaglia in \cite{G},
who also established the existence
and $H^2$ regularity of global attractors thus showing compactness 
at the infinite time. Further works by Moise, Rosa, Goubet, and Laurencot
lowered the regularity of the force and showing infinite
time compactness in periodic setting as well \cite{Go,GM,GR,L,MR}.

The existence and uniqueness of strong solutions of the stochastic KdV
equation on the domain $\dbR$ with additive noise is established in
\cite{DD1}. The authors provide estimates on the solutions of the
linear KdV equation and use these estimates to show local in time
existence of solutions for the nonlinear equation. Then using the
estimates in $H^1(\dbR)$, the authors show global existence of strong
solutions. We also mention that the problem was also studied in
\cite{P} on weighted Sobolev spaces.

The paper is organized as follows. In Section~\ref{sec2}, we introduce
the main notation, while Section~\ref{sec4} contains the main results.
Section~\ref{sec5} contains the main steps of the proof, including the
Feller
property and the asymptotic compactness. The Appendix
contains the proof of the convergence of norms,
which is the crucial step in showing compactness in
$L^{2}({\mathbb R})$ and 
$H^{1}({\mathbb R})$.

\startnewsection{The notation and the main result}{sec2}
\subsection{The stochastic Korteweg-de Vries equation}
Fix a stochastic basis $(\Omega,{\mathbb G},\{{\mathbb G}_t\}_{t\geq 0},\dbP)$.
With $(e_i)_{i\in \dbN}$, an orthonormal basis of $L^2(\dbR)$,
consisting of smooth compactly supported elements  
and $(\beta_i)_{i\in\dbN}$, a sequence of mutually independent 
one dimensional Brownian motions, denote by
  \begin{equation} W(t)=\sum_{i\in \dbN}\beta_i (t) e_i
  \end{equation}
a cylindrical Wiener process on $L^2(\dbR)$.
Consider the stochastic weakly damped Korteweg-de Vries equation
  \begin{equation}
    du + (\partial^3_x u +u\partial_x u +\lambda u)dt
    =
    f dt+
    \Phi dW(t),
   \label{kdv-intro}
  \end{equation}
where $\lambda>0$,
with the initial condition 
  \begin{equation}\label{kdv-initial}
 u(x,0)=u_0 (x),\quad x\in \dbR
   \period
  \end{equation}

\subsection{Notation}
For functions $u,v\in L^2(\dbR)$, 
denote by $\Vert  u\Vert_{L^2}$ the $L^2(\dbR)$ norm of $u$ and by $(u,v)$,
the
$L^2$-inner product of $u$ and $v$.
For a Banach space $B$ and with $T>0$ and $p>0$, denote by $L^p([0,T];B)$ the space of functions from $[0,T]$ into $B$ with integrable $p$-th power over $[0,T]$ and by $C([0,T];B)$ the set of continuous functions from $[0,T]$ into $B$. 

The Fourier transform (resp.\ the inverse Fourier transform) of $u$ is denoted by $ {\cal F} (u)$ (resp.\ $ {\cal F}^{-1} (u)$).  The Sobolev space $H^\sigma (\dbR)$ is the set of real functions $u$ verifying 
  \begin{equation}\Vert u\Vert^2_{H^\sigma}=\int_\dbR (1+|\xi|^2)^{\fractext{\sigma}{2}}|{\cal F} (u)(\xi)|^2 \,d\xi<\infty
   \end{equation}
%Also, let
 %\begin{equation}
 %   J_\sigma u= {\cal F}^{-1}\left((1+|\xi|^2)^{\fractext{\sigma}{2}} {\cal F} (u)(\xi)\right)
 %  \period
 %  \label{EQ07}
 % \end{equation}
and ${\cal B} (H^1(\dbR) )$ is the set of Borel measurable subsets of $H^1(\dbR)$.% and by $M_{2,4}(H^1(\dbR))$ the set of probability measures $\nu$ on $H^1(\dbR)$ such that 
 % \begin{equation}\int_{H^1} (\Vert u\Vert_{H^1}^2+\Vert u\Vert_{L^2}^4) \,\nu(du)<\infty.\end{equation}

For a Hilbert space $H$, we write ${\rm HS}(L^2,H)$ for the space of linear operators $\Phi$ from $L^2(\dbR)$ into $H$ with finite Hilbert-Schmidt norm 
  \begin{equation}
     \Vert\Phi\Vert_{{\rm HS}(L^2,H)}
        =
         \left(
          \sum_{i=1}^\infty \Vert\Phi e_i\Vert_H^2
         \right)^{1/2}
   \period
  \end{equation}
%and $W$ a cylindrical Wiener process on $L^2(\dbR)$ under the stochastic basis.  
Similarly to functional spaces, for $p>0$ we denote by $\dbL^p (\Omega,B)$ the space of random variables with values in $B$ and finite $p$-th moment.

\subsection{Well-posedness of the equation}

The equation \eqref{kdv-intro} was studied in \cite{DD1} in the case
$\lambda=0$. 
The arguments carry over to the case $\lambda>0$ with only slight modifications.

For all $\lambda \in \dbR$, we denote by $\{U_\lambda(t)\}_{t\geq 0}$ the solution operator of the partial differential equation
  \begin{equation}
 \label{linear-pde}
 du(t)+(\partial_x^3 u +\lambda u)dt=0
   \period
  \end{equation}
Note that $U_\lambda(t)=e^{-\lambda t}U_0 (t)$. We then write the equation \eqref{kdv-intro} in the mild form
  \begin{equation}\label{kdv-mild}
u(t)=U_\lambda(t) u_0-\int_0^t U_\lambda(t-s) u(s)\partial_x u(s)\,ds+\int_0^t U_\lambda (t-s) f \,ds+\int_0^t U_\lambda (t-s) \Phi \,dW(s)
   \period
\end{equation} 
Throughout the paper we assume that
  \begin{equation}
   f\in H^3(\dbR)
   \label{EQ00a}
  \end{equation}
and
  \begin{equation}
   v\to (v,f)\mbox{ is continuous in }L^2_{\rm loc}(\dbR)
   \period
   \label{EQ00b}
  \end{equation}
We also require
  \begin{equation}
   \Phi\in {\rm HS}(L^2(\dbR);H^{3+}(\dbR))
   \period
   \label{EQ00c}
  \end{equation}
%\begin{assumption}\label{assumption-existence}
% 1) $\lambda>0$\nonumber\\
% $\Phi\in {\rm HS}(L^2(\dbR);H^{\tilde \sigma}(\dbR))$ with $\tilde \sigma>1$.
%\end{assumption}
By ${\rm HS}(L^2(\dbR);H^{3+}(\dbR))$ we mean ${\rm HS}(L^2(\dbR);H^{\sigma}(\dbR))$ for some $\sigma>3$.  
Recall that $u$ is a mild solution of \eqref{kdv-intro} if $u$ verifies \eqref{kdv-mild} for all $t\geq 0$, $\dbP$ a.s. 
The following statement
addresses existence and uniqueness of solutions.

%We also need the following assumptions:

%\begin{definition}\label{definition-martingale-solution}
% For a given initial condition $u_0\in\dbL^2(\dbR)$, a martingale
% solution of {\eqref{kdv-intro}} is a stochastic basis $(\Omega^{u_0},\dbF^{u_0}, \dbP^{u_0})$, a filtration $\{\dbF^{u_0}_s\}_{s\geq 0}$ on this basis, a $\{\dbF^{u_0}_s\}_{s\geq 0}$-Wiener process $W^{u_0}$  and a $\{\dbF^{u_0}_s\}_{s\geq 0}$-adapted process $u$ such that the couple $(u,W^{u_0})$ verifies \eqref{kdv-mild}. 

% We say that the uniqueness of martingale solutions hold at $u_0$ if for 2 solutions  $(\Omega^{u_0},\dbF^{u_0}, \dbP^{u_0},\{\dbF^{u_0}_s\}_{s\geq 0},W^{u_0},u)$, $(\tilde \Omega^{u_0},\tilde \dbF^{u_0}, \tilde \dbP^{u_0},\{\tilde \dbF^{u_0}_s\}_{s\geq 0},\tilde W^{u_0},\tilde u)$, the distributions of $u$ and $\tilde u$ are equal.
%\end{definition}
%Note that, unlike the strong solutions, the stochastic basis and the Wiener process might% depend on the initial condition. Another main difference with strong solution is the fact that there might not be a measurable mapping 
%   \begin{equation}(t,\{W^{u_0}\}_{s\in[0,t]})\to u_t.
%   \end{equation}
%We only know that $W^{u_0}$ is a $\{\dbF^{u^0}_s\}$-Wiener process and $u$ us $\{\dbF^{u^0}_s\}$-adapted.

\coe
\begin{Theorem}\label{existence-solution}
 Assume that $u_0\in L^2 (\Omega; H^1(\dbR)) \cap L^4 (\Omega; L^2(\dbR)) $ is ${\cal G}_0$-measurable. Then there exists a unique mild solution of \eqref{kdv-intro} with paths almost surely in $C([0,\infty);H^1(\dbR))$ and with $u\in L^2 (\Omega; L^\infty(0,T ;H^1(\dbR)))$ for all $T>0$. 
  Additionally, if $u_0\in L^2 (\Omega; H^3(\dbR))$ then $u\in L^2 (\Omega; L^\infty(0,T ;H^3(\dbR)))$ for all $T>0$. 
\end{Theorem}
\cob
The theorem follows as in \cite{DD1} (Theorem 3.1 and Lemma 3.2) which treats the case $\lambda=0$ and it is thus omitted. The inclusion in $C([0,\infty);H^1(\dbR))$ is not explicitly mentioned in \cite{DD1}. However with 
the assumption 
%$\Phi\in {\rm HS}(L^2(\dbR);H^{3}(\dbR))$
\eqref{EQ00c}, we can use Theorem~3.2 and Proposition~3.5 of 
\cite{DD1} to obtain this inclusion.

%\begin{Remark}
%{\rm We require $u_s\in H^3(\dbR)$ to be able to apply Ito's lemma on $L^p$ norms of $u_s$.}
%\end{Remark}

\subsection{The Semigroup}
Let $u_0\in H^1(\dbR)$ be a deterministic initial condition, and let
$u$ be the corresponding solution of \eqref{kdv-intro}. For all $B\in {\cal B} (H^1(\dbR) )$ we define the transition probabilities of the equation by 
  \begin{equation}\label{transition-prob}
    P_t (u_0,B)=\dbP(u_t\in B)
   \period
  \end{equation}
For any function $\xi\in C_b(H^1;\dbR)$ and for $t\geq 0$ we denote 
  \begin{equation}\label{definition-psi}
    P_t \xi(u_0) =\dbE\left[\xi(u_t)\right]=\int_{H^1} \xi(u)P_t(u_0, du)
   \period
  \end{equation}
%
%Note that the pathwise uniqueness of Theorem~\ref{existence-solution} implies that for all $t\geq 0$ the distribution of $u_t$ only %depends on the distribution of $u_0$. 
%Thus we may define the mapping
 % \begin{equation} 
  % [0,\infty) \times M_{2,4}(H^1)
   %\to 
  % M_{2,4}(H^1)
  %\end{equation}
%by 
  %\begin{equation}
   %\Psi (t,\mu)
    % ={\cal L}(u_t)
 % \end{equation}
%where $u$ is any solution of the equation \eqref{kdv-intro} with initial condition having distribution $\mu$ over any probability basis.
%Here ${\mathcal L}(X)$ denotes the law of the random variable~$X$.

%Note that for all $t> 0$ the mapping $\mu\mapsto \Psi (t,\mu)$ is continuous with respect to the topology of weak convergence if and only if the semi-group $P_t$ is Feller. 

\startnewsection{The main results}{sec3}
We shall rely on the Krylov-Bogoliubov procedure (cf.~\cite[Corollary~3.1.2]{DZ}) to show the existence of invariant measures for the semigroup. 
The following statement is our main result.
 
\coe%%%out
\begin{Theorem}
 \label{theo-existence-invariant}
Suppose $\lambda>0$, and 
assume that $ f$ and $\Phi$ verify 
%Assumption~\ref{assumption-existence}
\eqref{EQ00a}--\eqref{EQ00c}.
Then there exists an invariant measure of the equation \eqref{kdv-intro}.  
\end{Theorem}
\cob%%%out

The proof is based on the following two lemmas.
The first lemma states that the Feller property holds for the stochastic damped KdV equation.

\coe
\begin{Lemma}{(Feller Property)}
\label{theo-feller}
 Under the assumptions of Theorem~\ref{theo-existence-invariant}, the semigroup 
$P_t$ is Feller on $H^1 (\dbR)$. Namely, for $\xi\in C_b(H^1,\dbR)$ and with $u_0^1, u_0^2, \ldots\in H^1(\dbR)$ satisfying $\Vert u^n_0-u_0\Vert_{H^1}\to 0$ as $n\to\infty$, 
where $u_0\in H^1(\dbR)$, the convergence 
 \begin{equation}P_t \xi(u_0^n)\to P_t \xi(u_0)
   \comma n\to\infty
 \end{equation} 
holds for all $t\ge0$.
\end{Lemma}
\cob

The second lemma asserts tightness
of averages originating from the initial
datum $u_0=0$.

\coe
\begin{Lemma}{(Tightness)}
 \label{lemma-tightness}
Under the assumptions of Theorem~\ref{theo-existence-invariant}, the family 
of measures
on $H^{1}({\mathbb R})$ 
  \begin{equation}
    \mu_n(\cdot)=\frac{1}{n}\int_0^n P_t(0,\cdot)\,dt   
   \comma n=1,2,\ldots
   \label{EQ02}
  \end{equation}
is tight.
\end{Lemma}
\cob

The proof of Lemma~\ref{theo-feller} is provided in
Section~\ref{secfeller},
while the proof of tightness is given in 
Section~\ref{secasy}.

\startnewsection{Proofs}{sec4}

%\subsection{Proof of tightness}

\subsection{Uniform bounds in $L^2(\dbR)$}

The next statement establishes
bounds on
$\dbE[\Vert u(t) \Vert^2_{L^2}]$
which are  uniform 
in $t$.

\coe
\begin{Lemma}\label{main-bounds}
Under the assumptions of Theorem~\ref{theo-existence-invariant}, 
there exists a sequence $\{C_k\}_{k\geq 1}$ depending on $f$, $\Phi$, 
and $\lambda$ such that
 \begin{align}
   \sup_{t\geq 0} \dbE \left[ \Vert u(t)\Vert_{L^2}^{2k}\right]\leq C_k(\lambda,\Phi)(\dbE\left[\Vert u_0\Vert_{L^2}^{2k}\right]+1)
   \label{EQ54}
 \end{align} 
holds for all $k\in{\mathbb N}$ for which
$\dbE\left[\Vert u_0\Vert_{L^2}^{2k}\right]<\infty$.
\end{Lemma}
\cob

From here on, we shall consider $\lambda$, $f$,
and $\Phi$ fixed and thus
omit indicating the dependence of the constants on these
two quantities.

\begin{proof}[Proof of Lemma~\ref{main-bounds}]
We define 
  \begin{equation}
   \sigma_n
      =\left\{t\geq 0:\int_0^t \Vert u_s\Vert^2_{L^2}\,ds\geq n\right\}   
   \label{EQ04}
  \end{equation}
and
apply Ito's lemma to $\Vert u(t)\Vert_{L^2}^2$ in order to obtain 
  \begin{equation}\label{bound-u1}
    \Vert u_t\Vert_{L^2}^2+2\lambda\int_0^t \Vert u_s\Vert_{L^2}^2\,ds 
     =\Vert u_0\Vert_{L^2}^2
       +t\Vert\Phi\Vert^2_{{\rm HS}(L^2,L^2)}
       +  \int_0^t 2 (u_s,f)\, ds
       +2 M_t
  \end{equation}
for $0\leq t\leq \sigma_n$,
where 
  \begin{equation}
    M_t=\int_0^t \sum_i\int_\dbR u_s \Phi e_i \,dx \,d\beta^i_s   
   \period
  \end{equation}
We compute 
  \begin{align}
   \dbE[M_{t\wedge \sigma_n}^2]
    &= \dbE\left[\int_0^{t\wedge \sigma_n}\sum_i (u_s,\Phi e_i)^2
    ds\right]
   \leq \Vert \Phi \Vert_{HS(L^2,L^2)}^2 \dbE\left[\int_0^{t\wedge
   \sigma_n} \Vert u_s\Vert^2  ds\right]
   \nonumber\\&
   \leq 
     n\Vert \Phi \Vert_{HS(L^2,L^2)}^2 t<\infty,
  \end{align}
%Using the  martingale and $\epsilon$-Young inequalities, we get
  %\begin{equation}\label{control-m}
   %\dbE\left[\sup_{s\in[0,t]} |M_t|\right]\leq \dbE\left[ \left(\int_0^T \Vert\Phi^* u(s)\Vert_{L^2}^2 \,ds\right)^{\fractext{1}{2}} \right]\leq \dbE\left[\sup_{s\in[0,t] }\Vert u(s)\Vert_{L^2}^2\right]+Ct\Vert\Phi\Vert^2_{{\rm HS}(L^2,L^2)}.
 %  \end{equation}
%\nts{why factor 3?}
%By \cite[Theorem~3.1]{DD1}, we have 
  %\begin{equation}
   %\dbE\left[\sup_{s\in[0,t] }\Vert u(s)\Vert_{L^2}^2\right]<\infty
   %\label{EQ03} 
  %\end{equation}
and thus $M_{t\wedge \sigma_n}$ is  a martingale. 
Taking the expectation of  both sides of the equation \eqref{bound-u1}
at $t\wedge \sigma_n$, we get
  \begin{align}
      &\dbE
       \Vert u(t\wedge \sigma_n)\Vert_{L^2}^2
        +2\lambda\dbE\left[\int_0^{t\wedge \sigma_n} \Vert u(s)\Vert_{L^2}^2\,ds \right]
      \nonumber\\&\indeq
      =\dbE\Vert u_0\Vert_{L^2}^2
           +\dbE[t\wedge \sigma_n]\Vert\Phi\Vert^2_{{\rm HS}(L^2,L^2)}
      + 2\dbE \left[\int_0^{t\wedge \sigma_n} (u_s,f) ds \right]
   \period
   \label{EQ05bis}
  \end{align}
Therefore, we obtain an upper bound 
   \begin{align}
      \dbE
       \Vert u(t\wedge \sigma_n)\Vert_{L^2}^2
        +\lambda\dbE\left[\int_0^{t\wedge \sigma_n} \Vert u(s)\Vert_{L^2}^2\,ds \right]
      &\leq \dbE\Vert u_0\Vert_{L^2}^2+t\Vert\Phi\Vert^2_{{\rm HS}(L^2,L^2)}+ \frac{t\Vert f\Vert^2_{L^2}}{\lambda}
  \end{align}
which is uniform in $n$.
Note that we have
  \begin{equation}
    \int_0^{t\wedge \sigma_n} \Vert u(s)\Vert_{L^2}^2\,ds
      =
          \int_0^t  \Vert u(s)\Vert_{L^2}^2\,ds \mathbf{1}_{\{\sigma_n>t\}}
           +n \mathbf{1}_{\{\sigma_n\leq t\}}
   \label{EQ06}
  \end{equation}
from where
  \begin{align}\label{prop-sigma-n}
\lambda n \dbP(\sigma_n \leq t)\leq \lambda\dbE\left[\int_0^{t\wedge \sigma_n} \Vert u(s)\Vert_{L^2}^2\,ds \right]\leq C(t)
   \period
\end{align}
Taking the limit as $n$ goes to infinity, we see that the stopping time $\sigma^*= \lim_{n\to\infty } \sigma_n$ verifies  $\dbP(\sigma^*=\infty)=1$.
We now return to the equality \eqref{EQ05bis}. Using the dominated
convergence theorem and the fact that $\Vert u(s)\Vert_{L^2}$ is
continuous, we obtain
  \begin{align}
      \dbE
       \Vert u(t)\Vert_{L^2}^2
        +2\lambda\dbE\left[\int_0^{t} \Vert u(s)\Vert_{L^2}^2\,ds \right]
      &=\dbE\Vert u_0\Vert_{L^2}^2
           +t\Vert\Phi\Vert^2_{{\rm HS}(L^2,L^2)}
      + 2\dbE \left[\int_0^{t} (u_s,f) ds \right]
   \period
  \end{align}  
We differentiate the previous equality and 
use the $\epsilon$-Young's inequality to control $\dbE (u_s,f)$ and obtain
 \begin{align}
     \frac{d }{dt}\dbE
       \Vert u(t)\Vert_{L^2}^2
        +\lambda\dbE \Vert u(t)\Vert_{L^2}^2\leq \Vert\Phi\Vert^2_{HS(L^2,L^2)} +\frac{\Vert f\Vert_{L^2}^2}{\lambda}
   \period
   \label{EQ55}
  \end{align}
Solving the resulting
equation, we get
  \begin{align}
     \dbE\bigl[\Vert u(t)\Vert_{L^2}^2\bigr]
       &\leq  e^{- \lambda t}\dbE\left[\Vert u_0\Vert_{L^2}^2\right]
        +
         \left(
             \Vert\Phi\Vert^2_{{\rm HS}(L^2;L^2)}
             + \frac{\Vert f\Vert_{L^2}^2}{\lambda}
         \right)
           \int_0^t e^{-\lambda (t-s)}  \,ds 
%          \nonumber\\&
     \leq C(\dbE\left[\Vert u_0\Vert_{L^2}^2\right]+1),
  \end{align}
where the constant $C$ depends on $\Phi$ 
only through $\Vert\Phi\Vert^2_{{\rm HS}(L^2;L^2)}$.

In order to use the induction for $k\geq 1$, we need to control $\dbE\left[\sup_{0\leq s \leq t}\Vert u(s)\Vert^2_{L^2}\right]$. To achieve this, we return to \eqref{bound-u1} and obtain 
\begin{align}
\sup_{s\in [0,t]}\Vert u(s)\Vert^2_{L^2}\leq \Vert u_0\Vert^2_{L^2} +C(t) + C\sup_{s\in[0,t]}|M_s|
   \period
\end{align}
Then, by the Burkholder-Davis-Gundy inequality 
  \begin{align}
   \dbE[\sup_{s\in[0,t]}|M_s|]
     &\leq C\dbE\left[\left(\int_0^t \sum_i (u(s),\Phi e_i)^2 ds\right)^{1/2}\right]\nonumber\\
&\leq  C \left(\dbE\left[\int_0^t \Vert u(s)\Vert^2_{L^2}ds\right]\right)^{1/2}
%\nonumber\\&
  \leq C(t),
  \end{align}
which then gives $\dbE[\sup_{s\in[0,t]}\Vert u(s)\vert^2_{L^2}]
\leq C(t)$.

In order to use induction, suppose that for some $k\in{\mathbb N}$
  \begin{align}
 &\sup_{t\geq 0} \dbE \left[ \Vert u(t)\Vert_{L^2}^{2k}\right]\leq C(\dbE\left[\Vert u_0\Vert_{L^2}^{2k}\right]+1)<\infty
  \end{align}
and
  \begin{align} \dbE[\sup_{s\in[0,t]}\Vert u(s)\Vert^{2k}_{L^2}]\leq C_k(t)\label{iteration-k}
   \period
  \end{align}
Recall that, but the assumptions of the theorem,
  \begin{align}
    &\dbE \left[\Vert u_0\Vert_{L^2}^{2(k+1)}\right]<  \infty
   \period
  \end{align}
Denote $X_t=\Vert u(t)\Vert_{L^2}^2$ and, similarly to the previous case, let
  \begin{equation}
   \sigma_n
     =\inf\left\{t\geq 0: \int_0^t X^{k+1}_s ds \geq n\right\}
   \period
   \label{EQ56}
  \end{equation}
Applying Ito's lemma leads to
  \begin{align}\label{ito-for-power}
    dX_t^{k+1}&=(k+1)X_t^k
    \bigl(-2\lambda X_t dt +\Vert\Phi\Vert^2_{{\rm HS}(L^2,L^2)} dt 
     +2(u(t),f)dt
     +2dM_t
    \bigr)
      \nonumber\\&\indeq
      +2 ( k+1)k X_t^{k-1} \Vert\Phi^* u(t)\Vert_{L^2}^2 dt
   \nonumber\\&
   \leq -\lambda (k+1)X_t^{k+1} dt+ C (X_t^k+1) dt +2(k+1) X_t^{k} dM_t
   \period
  \end{align}
The quadratic variation of the stochastic integral is proportional to 
  \begin{align}\label{control-quad-var-power}
    & \int_0^{T\wedge \sigma_n} X_t^{2k} d\langle M\rangle_t 
%      \leq  \int_0^T X_t^{2k} \sum_i (u(t),\Phi e_i)^2 dt
      \leq C\int_0^{T\wedge \sigma_n}  X_t^{2k} \sum_i (u(t),\Phi e_i)^2 dt
%     \nonumber\\&\indeq
    \leq C  \int_0^{T\wedge \sigma_n}  X_t^{2k+1} \,dt \leq C n \sup_{t\in T\wedge \sigma_n} \Vert u(t)\Vert^{2k}_{L^2} \,dt
  \end{align}
where the brackets denote the quadratic variation and the last term is integrable due to \eqref{iteration-k}.
Thus $\int_0^{t\wedge \sigma_n} X_s^k dM_s$ is a square integrable
martingale. We take the expectation of \eqref{ito-for-power} and obtain
an upper bound 
   \begin{align}\label{ito-for-power2}
    \dbE[X_{t\wedge\sigma_n}^{k+1}]+2(k+1)\lambda \dbE\left[\int_0^{t\wedge\sigma_n}X_s^{k+1}ds\right]\leq 
     C(t)
  \end{align}
which is uniform in $n$.
Similarly to \eqref{prop-sigma-n}, we have 
   \begin{equation}\frac{1}{\sqrt{n}}\left(\int_0^{t\wedge\sigma_n}
  X_s^{2k+1}ds\right)^{1/2}\geq \mathbf{1}_{\{\sigma_n\leq t\}}
   \period
  \end{equation}
Thus
 \begin{align}
 \dbP(\sigma_n\leq t)&\leq \frac{1}{\sqrt{n}} \dbE\left[\left(\int_0^{t\wedge\sigma_n} X_s^{2k+1}ds\right)^{1/2}\right]\nonumber\\
 &\leq  \frac{1}{\sqrt{n}} \dbE\left[\left(\sup_{s\in[0,t]}X_s^k\right)^{1/2}\left(\int_0^{t\wedge\sigma_n} X_s^{k+1}ds\right)^{1/2}\right]\nonumber\\
  &\leq  \frac{C}{\sqrt{n}}\left( \dbE\left[\sup_{s\in[0,t]}X_s^k\right]+\dbE\left[\int_0^{t\wedge\sigma_n} X_s^{k+1}ds\right]\right)\leq\frac{C(t)}{\sqrt{n}}
 \end{align}
 which converges to $0$ as $n\to\infty$ 
due to \eqref{ito-for-power2} and the inductive assumption.
Thus the stochastic integral in \eqref{ito-for-power} is a martingale
  on $[0,\infty)$.
Taking the expectation of \eqref{ito-for-power}, we get
 \begin{align}
 d\dbE[X_t^{k+1}]+\lambda (k+1)\dbE[X_t^{k+1}]\,dt\leq C\,dt
 \end{align}
which implies
  \begin{align} 
 \dbE[X_t^{k+1}]
    &\leq e^{-\lambda (k+1)t}\dbE[X_0^{k+1}] 
    +C\int_0^t e^{-\lambda(k+1)(t-s)} \,ds
       \nonumber\\&
       \leq C \left(\dbE\left[\Vert u_0\Vert_{L^2}^{2(k+1)} \right]+1\right)
   \period
  \end{align}

In order to complete the proof, we 
also need to control $\sup_{t\in [0,T]} X_t^{k+1}$. 
Note first that $X_t\geq 0$.
From \eqref{ito-for-power}, we get
  \begin{equation}\sup_{t\in [0,T]} X_t^{k+1}\leq C(T)\left(X_0^{k+1}+  \sup_{t\in[0,T]}X_t^k\right) +\sup_{t\in[0,T]} \left|\int_0^t 2(k+1)X_s^{k}\,dM_s\right|
   \period
 \end{equation}
We use the BDG inequality to bound the expectation of the last term with the expectation of the square root of the quadratic variation of the stochastic integral computed in \eqref{control-quad-var-power} to obtain 
  \begin{equation}
    \dbE\left[\sup_{t\in [0,T]} X_t^{k+1} \right]\leq C(T)
  \end{equation}
and the proof is complete.
\end{proof}

\nnewpage
\subsection{Uniform bounds in $H^1(\dbR)$}

We now present certain uniform bounds on the $L^2$ norm of $\partial_x u $. 
Denote by 
  \begin{equation}\label{invariant-I}
   I(v)=\int_\dbR \left(\partial_x v(x)^2-\frac{v(x)^3}{3}\right)dx
   \comma v\in H^1(\dbR)
  \end{equation}
the second invariant of the deterministic KdV equation.
For ease of notation we define  $$\alpha(t)=\frac{\lambda}{3}\int u(t,x)^3 dx+\Vert\partial_x \Phi\Vert^2_{{\rm HS}(L^2,L^2)} 
              - \sum_i \int u(t,x)|(\Phi e_i )(x)|^2 \,dx
              +2(\partial_x u(t),\partial_x f)-(u^2(t),f).$$

\coe
\begin{Lemma}\label{lemma-invariant-I}
The evolution of $I(u(t))$ is given by
  \begin{align}
  dI(u(t))+2\lambda I(u(t)) dt 
     =\alpha(t) dt
  +2 (\partial_x u(t), \partial_x \Phi dW(t))-(u^2(t),\Phi dW_t)
 \label{evolution-I}
   \period
  \end{align}
Moreover,
  \begin{equation}\sup_{t\geq 0}
      \dbE\left[\Vert\partial_x u(t)\Vert_{L^2}^{2k}
           \right]
      \leq C\left(\dbE\left[\Vert u_0\Vert^{2k}_{H^1}+\Vert u_0\Vert_{L^2}^{4k}\right]+1\right)\period
   \label{EQ09}
  \end{equation}
for $k=1$ or $2$, where $C$  is a constant.
\end{Lemma}
\cob

\begin{proof}[Proof of Lemma~\ref{lemma-invariant-I}]
The identity \eqref{evolution-I}
follows by Ito's formula as in~Lemma 3.3 of \cite{DD1}.
The quadratic variations of the stochastic integrals 
in \eqref{evolution-I}
equal
  \begin{align}\label{quad-var-H1}
   \bigl\langle (\partial_x u(t), \partial_x \Phi dW(t))\bigr\rangle&=\sum_i   (\partial_x u(t), \partial_x \Phi e_i)^2 dt \leq \Vert\partial_x u(t)\Vert_{L^2}^2 \Vert \Phi\Vert^2_{{\rm HS}(L^2,H^1)}dt
  \end{align}
(the brackets denoting the quadratic variation)
and
  \begin{align}
 \bigl\langle (u^2(t),\Phi dW_t)\bigr\rangle&=\sum_i (u^2(t),\Phi e_i)^2 dt\leq \Vert u(t)\Vert_{L^2}^4 \sum_i \Vert\Phi e_i\Vert_{\infty}^2 dt
   \nonumber\\&
     \leq \Vert u(t)\Vert_{L^2}^4 \sum_i \Vert\Phi e_i\Vert_{L^2} \Vert\partial_x \Phi e_i\Vert_{L^2} dt \leq \Vert u(t)\Vert_{L^2}^4 \Vert\Phi\Vert^2_{{\rm HS}(L^2,H^1)}dt
   \period
   \label{EQ14}
  \end{align}
Theorem~\ref{existence-solution}, combined with the previous inequalities,
shows that the stochastic integrals define martingales. 

Now, we estimate $\alpha$. 
Using Agmon's inequality
  \begin{equation}
 \Vert u\Vert_{L^\infty}\leq \Vert u\Vert_{L^2}^{\fractext{1}{2}} \Vert\partial_x u \Vert_{L^2}^{\fractext{1}{2}}
   \period
   \label{EQ01}
  \end{equation}
we obtain
  \begin{equation}\label{estimate_alpha}
  |\alpha(t)|\leq C\left(\Vert u\Vert_{L^2}^{5/2}\Vert \partial_x u\Vert_{L^2}^{1/2}+\Vert u\Vert_{L^2}+\Vert \partial_x u\Vert_{L^2}+\Vert u\Vert_{L^2}^{3/2}\Vert \partial_x u\Vert_{L^2}^{1/2}\right)
  \end{equation}
Then we use the $\epsilon$-Young inequality and obtain 
  \begin{equation}
|\alpha(t)|    \leq   {\lambda}\Vert\partial_x u\Vert_{L^2}^2 + C(\Vert u\Vert_{L^2}^{10/3}+1)\leq  {\lambda}\Vert\partial_x u\Vert_{L^2}^2 + C(\Vert u\Vert_{L^2}^{4}+1)
   \period 
  \end{equation}
Inserting these inequalities into \eqref{evolution-I}
and taking expectation,
we obtain
  \begin{align}\label{bound-I-diff}
     & d \dbE[I(u(t))]+2\lambda \dbE[I(u(t))]dt 
      \leq  {\lambda } \dbE[\Vert\partial_x u(t)\Vert_{L^2}^2] dt + C\left(1+\dbE\left[\Vert u_0\Vert_{L^2}^{4}\right]\right)dt
   \period
  \end{align}
For all $v\in H^1(\dbR)$ the inequalities in \cite[Section~5]{R} give
  \begin{equation}\label{control-derivative}
 \frac{2}{3}\Vert\partial_x v\Vert_{L^2}^2 -C \Vert v\Vert_{L^2}^{\fractext{10}{3}}\leq I(v)\leq \frac{4}{3}\Vert\partial_x v\Vert_{L^2}^2+C\Vert v\Vert_{L^2}^{\fractext{10}{3}}
   \period
  \end{equation}
Using the left inequality, we get
  \begin{align}
   \dbE[I(u(t))]
    &\leq e^{-\lambda t/2}\dbE[I(u_0)]+ C\int_0^t e^{-\lambda (t-s)/2 } \left(1+\dbE[\Vert u_0\Vert_{L^2}^4]\right)\,ds 
    \nonumber\\&
    \leq C(\dbE[I(u_0)]+\dbE[\Vert u_0\Vert_{L^2}^4]+1)
   \period
  \end{align}
By the second part of \eqref{control-derivative}  we
obtain 
  \begin{equation}
    \sup_{t\geq 0 }\dbE[\Vert\partial_x u(t)\Vert_{L^2}^2]\leq C(\dbE\left[I(u_0)+\Vert u_0\Vert_{L^2}^4\right]+1)
   \period
  \end{equation}
Finally, combining all the inequalities above we conclude
  \begin{equation}\label{uniform-bounds-H1}
 \sup_{t\geq 0} \dbE[\Vert u(t)\Vert^2_{H^1}]\leq C(\dbE\left[\Vert u_0\Vert^2_{H^1}+\Vert u_0\Vert_{L^2}^4\right]+1)
  \end{equation}
which gives \eqref{EQ09}.

In order to obtain \eqref{EQ09} for $k=2$, we apply Ito's Lemma to $I^2(u(t))$ and get
 \begin{align}
    &
     dI^2(u(t))+4\lambda I^2(u(t)) dt 
    \nonumber\\&\indeq
     =2I(u(t))\alpha(t) dt +d\tilde M_t
     +\sum_i\Bigl(2(\partial_x u(t), \partial_x \Phi e_i)-(u^2(t),\Phi e_i)
            \Bigr)^2 dt,
  \end{align}
where 
\begin{align*}
              d\tilde M_t=2 I(t)\sum_i \left\{2(\partial_x u(t), \partial_x \Phi e_i)-(u^2(t),\Phi e_i)\right\}dB^i_t\period
     \end{align*}
Similarly to the previous case, we have
$$2I(u(t))\alpha(t)\leq2  \lambda I^2(u(t))+C(1+ \Vert \partial_x u(t)\Vert^2_{L^2}+ \Vert u(t)\Vert_{L^2}^{20/3}).$$
We also estimate the quadratic variation term as
  \begin{align}
    &\sum_i\Bigl(2(\partial_x u(t), \partial_x \Phi e_i)-(u^2(t),\Phi e_i)
      \Bigr)^2 
    \nonumber\\&\indeq
     \leq 8\Vert \partial_x u(t)\Vert^2_{L^2}\Vert\partial_x \Phi\Vert^2_{{\rm HS}(L^2,L^2)}+ 2 \Vert u\Vert^4_{L^4} \Vert \Phi\Vert^2_{{\rm HS}(L^2,L^2)}
    \nonumber\\&\indeq
   \leq C( \Vert \partial_x u(t)\Vert^2_{L^2}+ \Vert u\Vert^6_{L^2} )
  \end{align}
where the right hand side is bounded in expectation. 

Finally, we compute the quadratic variation of $\tilde M$, 
\begin{align*}
d\langle \tilde M\rangle_t&=4I^2(u(t))\sum_i\left\{2(\partial_x u(t), \partial_x \Phi e_i)-(u^2(t),\Phi e_i)\right\}^2dt\\
&\leq C I^2(u(t))( \Vert \partial_x u(t)\Vert^2_{L^2}+  \Vert u(t)\Vert^4_{L^4})dt
\nonumber\\&
\leq CI^2(u(t))( \Vert \partial_x u(t)\Vert^2_{L^2}+ \Vert u\Vert^6_{L^2} ) .
\end{align*}
where we used \eqref{EQ01} and the $\epsilon$-Young inequality. For all $n\in{\mathbb N}$, we define the stopping time 
  \begin{equation}
   \tau_n=\inf\left\{t\geq 0 :\int_0^t I^2(u(s))ds\geq n
          \right\}   
   \label{EQ57}
  \end{equation}
 and 
  \begin{equation}
   \tau^*
    =\lim_n \tau_n
   \period   
   \label{EQ58}
  \end{equation}
Integrating the evolution of $I^2(u(t))$ we obtain for all $T>0$
  \begin{align*}
    &\sup_{t\in [0,T\wedge \tau_n]} I^2(u(t))-I^2(u(0))+4\lambda \int_0^t I^2(u(s))ds
    \nonumber\\&\indeq
     \leq \sup_{t\in [0,T\wedge \tau_n]} \tilde M_t +2\lambda \int_0^{T\wedge \tau_n} I^2(u(s))ds 
+C\int_0^{T\wedge \tau_n}\left( \Vert \partial_x u(s)\Vert^2_{L^2}+  \Vert u(s)\Vert_{L^2}^{20/3} \right) ds
\end{align*}
We now take the expectation and use the  Burkholder-Davis-Gundy inequality to obtain 
\begin{align*}
&\dbE\left[\sup_{t\in [0,{T\wedge \tau_n}]} I^2(u(t))-I^2(u(0))+4\lambda \int_0^t I^2(u(s))ds\right]
\nonumber\\&\indeq
 \leq 2\lambda\dbE\left[ \int_0^{{T\wedge \tau_n}}  I^2(u(s))ds \right]
+ C\dbE\left[\left(\int_0^{T\wedge \tau_n} I^2(u(t))(\Vert \partial_x u(t)\Vert^2_{L^2}+1) dt\right)^{1/2}\right] 
\nonumber\\&\indeq\indeq
+C\dbE\left[\int_0^{T\wedge \tau_n} \Vert \partial_x u(s)\Vert^2_{L^2}+  \Vert u(s)\Vert_{L^2}^{20/3} ds\right],
\end{align*}
We can use the $\epsilon$-Young inequality to have
$$\dbE\left[\left(\int_0^{T\wedge \tau_n} I^2(u(t))(\Vert \partial_x u(t)\Vert^2_{L^2}+1) dt\right)^{1/2}\right]\leq \lambda 
\dbE\left[\int_0^{T\wedge \tau_n}  I^2(u(s))ds\right] +C\dbE\left[\sup_{t\in[0,T]}\Vert\partial_x u(t)\Vert^2_{L^2}\right].$$
Using this inequality in the previous estimate, we obtain that for all $T>0$ we have
  \begin{equation}
     \dbE\left[\int_0^{T\wedge \tau_n}I^2(u(s))ds\right]<C   
   \label{EQ59}
  \end{equation}
for a constant independent of $n$ which implies that  $\tilde M$ is a martingale. 
Therefore,
\begin{align*}
 \dbE\left[I^2(u(t))\right]+2\lambda \int_0^t \dbE\left[I^2(u(s))\right]ds\leq&\dbE[I^2(u(0))]+ C\int_0^t \dbE\left[\Vert \partial_x u(s)\Vert^2_{L^2}+  \Vert u(s)\Vert_{L^2}^{20/3}\right] ds,
\end{align*}
where the function $s\mapsto  \dbE\left[\Vert \partial_x u(s)\Vert^2_{L^2}+  \Vert u(s)\Vert_{L^2}^{20/3}\right] $ is bounded. Proceeding as above, we obtain  $$\sup_{t\geq 0} \dbE\left[I^2(u(t))\right]<C\dbE\left[\Vert \partial_x u(0)\Vert^4+ \Vert u(0)\Vert_2^{20/3}+1\right]$$
and the lemma is established.
\end{proof}

In order to prove the Feller property,
we need the following improvement of our estimates.

\coe
\begin{Lemma}\label{finite-time-H1}
For all $R_0, T>0$ there exists a constant $C(R_0,T)$ such that 
we have
  \begin{align}
 &\dbE\left[\sup_{t\in[0,T]}\Vert u(t)\Vert^2_{H^1}\right]\leq C(R_0,T) 
  \end{align}
for all initial conditions $u_0\in H^1(\dbR)$ verifying $\Vert u_0\Vert_{H^1}\leq R_0$.
\end{Lemma}
\cob

\begin{proof}
We write \eqref{evolution-I} in the form
  \begin{align}
   I(u(t))&=I(u_0) +M_t-2\lambda \int_0^t\ I(u(s))  + \Vert\partial_x \Phi\Vert_{{\rm HS}(L^2,L^2)}^2ds
     \nonumber\\&\indeq
	+\int_0^t \biggr(\frac{\lambda}{ 3} \int u^3(s,x) \,dx
      -\sum\int_\dbR u(s,x) |(\Phi e_i)(x)|^2 \,dx
   \nonumber\\&\indeq\indeq\indeq\indeq\indeq\indeq\indeq
     +2(\partial_x u(s),\partial_x f)-(u^2(s),f)\biggr)\,ds
  \end{align}
where $M_t$ is the martingale term.
Then 
  \begin{align}
  & \dbE\left[\sup_{0\leq t\leq T} |I(u(t))| \right]
    \nonumber\\&\indeq
    \leq \dbE|I(u_0)| 
      +C \int_0^T \dbE\left[|I(u(s))|\right]\,ds 
   \nonumber\\&\indeq\indeq
   +C\int_0^T \dbE\left[ |u^3(s,x)| \,dx
   + \Vert\partial_x \Phi\Vert_{{\rm HS}(L^2,L^2)}^2 +\sum\int_\dbR |u(s,x)| |(\Phi e_i)(x)|^2 \,dx \right]\,ds
   \nonumber\\&\indeq\indeq
   +\dbE\left[\sup_{0\leq t\leq T}| M_t|\right]
  \nonumber\\&\indeq
  \leq\dbE\left[ | I(u_0)|+ C T (\Vert u_0\Vert^2_{H^1}+\Vert u_0\Vert_{L^2}^4+1)\right] +\dbE\left[\sup_{0\leq t\leq T} | M_t|\right]
   \period
  \end{align}
The Burkholder-Davis-Gundy inequality, together with
\eqref{quad-var-H1} and \eqref{EQ14}, gives the required bound for $\dbE\left[\sup_{0\leq t\leq T} | M_t|\right]$. 
%Define $J_t:=I_{t\wedge \tau}$. One can write the evolution of $J$ similar to \eqref{evolution-I}. $\tau$ is bounded, by stopping theorem the stochastic integrals have $0$ expectation. Therefore, one can obtain an inequality similar to \eqref{bound-I-diff}: 
%\begin{equation}
%\begin{split}
% d \dbE[J_t]+2\lambda \dbE[J_t]dt \leq  \left(\frac{2\lambda }{3} \dbE[\Vert\partial_x u(t\wedge \tau)\Vert_{L^2}^2] + C\left(1+\Vert u_0\Vert_{L^2}^4\right)\right)dt
%\end{split}
%\end{equation}
%and 
%\begin{equation}
% \sup_{t\geq 0} \dbE[\Vert u(t\wedge \tau)\Vert^2_{H^1}]\leq C(\Vert u_0\Vert^2_{H^1}+\Vert u_0\Vert_{L^2}^4+1)
%\end{equation}
%which in particularly gives that $\dbE[\Vert u(\tau)\Vert^2_{H^1}]\leq C(\Vert u_0\Vert_{L^2}^2_{H^1}+\Vert u_0\Vert_{L^2}^4+1)$.
\end{proof}

\nnewpage
\subsection{Proof of the Feller property}
\label{secfeller}
For all $T\geq 0$ we, as in \cite{DD1}, define
  \begin{align}
    X_1(T)
    &=\Bigl\{
      u\in C(0,T; H^{1}({\mathbb R}))\cap L^2({\mathbb R};L^{\infty}([0,T])):
    \nonumber\\&\indeq\indeq
        \nabla\partial_{x}u\in L^{\infty}({\mathbb R},L^2([0,T])),
         \partial_x u\in L^{4}([0,T]; L^{\infty}({\mathbb R}))
     \Bigr\}
   \label{EQ08}
  \end{align}
and denote
  \begin{equation}
    \overline u_T
     =\int_0^T U_\lambda(T-s)\Phi \,dW_s   
    \period    
   \label{EQ11}
  \end{equation}
For a stopping time $\tau$ we define a shifted process by
  \begin{align}\label{defn-shifted}
    \overline u^\tau_T&=0 \mbox{~~for }T\leq \tau
  \end{align}
and
  \begin{align}
  \overline u^\tau_T&=\int_\tau^T U_\lambda(T-s)\Phi \,dW_s 
   \comma T\geq \tau
   \period
  \end{align}

We start with the following auxiliary result.

\coe
\begin{Lemma}[\cite{DD1}]
\label{continuity-ubar}
Under the assumption
%assumptions \ref{assumption-existence} 
\eqref{EQ00c}, $\Vert\overline u\Vert_{X_1(s)}$ is a continuous process. 
\end{Lemma}
\cob

\begin{proof}
Since
$\overline u$ is a solution of a linear equation,
it is sufficient to prove that
%Thus to show the continuity of 
%$\Vert\overline u\Vert_{X_1(s)}$ it is enough to show that 
  \begin{equation}
   \Vert\overline u\Vert_{X_1(s)} \to 0
   \label{EQ20}
   \comma \dbP\mbox{-a.s. }
  \end{equation}
as $s\to 0$. Now,
${X_1(s)}$ is defined 
in \eqref{EQ08}
as the intersection of four spaces 
and thus
we need to show convergence to $0$ for all four norms. 

Note that $\Vert\overline u(s)\Vert_{H^1}$, for $s\in[0,T]$ is a continuous uniformly integrable 
semi-martingale 
and thus 
  \begin{equation}
    \dbE\left[\sup_{0\leq r\leq s} \Vert\overline u(r)\Vert_{H^1}\right]\to 0   
    \mbox{~~as~}s\to0
   \period
   \label{EQ21}
  \end{equation}
For the norms associated to $ L^2({\mathbb R};L^{\infty}([0,T]))$ and
$ L^{4}([0,T]; L^{\infty}({\mathbb R}))$, the convergence follows by
the monotone convergence theorem. The only issue is for the
$L^{\infty}({\mathbb R},L^2([0,T]))$ norm. 
In order to show convergence, we modify the proof of 
\cite[Proposition~3.3]{DD1} to obtain
  \begin{equation}
  \dbE\left[\sup_{x\in \dbR} \int_0^s |D\partial_x \overline u(r) |^2 dr\right]\leq C(\lambda, s)\Vert\Phi\Vert^2_{{\rm HS}(L^2, H^{3})}   
   \label{EQ22}
  \end{equation}
where $C(\lambda, s)\to 0$ as $s\to 0$, which completes the proof.
\end{proof}

%{\begin{proof}[Proof of Theorem~\ref{continuity-ubar}]
We 
now return to the proof of the Feller property. Fix $u_0\in H^{1}({\mathbb R})$ and $t>0$.
Also, let $\xi\in C_b(H^1,\dbR)$ and $\epsilon>0$.
Denote
$R_0=\Vert u_0\Vert_{H^1}+1$
and 
  \begin{equation}
   M=\sup_{v\in H^1(\dbR)}|\xi(v)|   
   \period
   \label{EQ28}
  \end{equation}

{\it Step 1}: For all $v_0$ such that $\Vert v_0-u_0\Vert_{H^1}\leq 1$
  with the associated solution $v$ of \eqref{kdv-intro}, Lemma~\ref{finite-time-H1} gives
  \begin{align}
    & \dbP\left(
            \max\left\{
                \sup_{s\in[0,t]}\Vert u(s)\Vert^2_{H^1},
                \sup_{s\in[0,t]}\Vert v(s)\Vert^2_{H^1}\right\}\geq R\right)
    \nonumber\\&\indeq
     \leq\frac{1}{R} \dbE\left[\sup_{s\in[0,t]}\Vert u(s)\Vert^2_{H^1}+\sup_{s\in[0,t]}\Vert v(s)\Vert^2_{H^1}\right]
   \nonumber\\&\indeq
 \leq\frac{C(R_0)}{R}
   \period
  \end{align}
(Note that $t>0$ is fixed and thus the dependence of all constants
on $t$ is not indicated.)
Fix $R>0$ so that the last term verifies 
  \begin{equation}
    \frac{C(R_0)}{R}\leq\frac{\epsilon}{6 M}   
   \period
   \label{EQ12}
  \end{equation}

{\it Step 2}:  Using results in \cite{DD1}, we obtain
a non-decreasing function $\CC$ 
such that we have the inequalities
  \begin{align}
    &\dbE\Vert\overline u\Vert_{X_1(T)}+ \left\Vert\int_0^\cdot U_\lambda (\cdot-r)f \,dr\right\Vert_{X_1(T)} 
      \leq \CC(T) %\Vert\Phi\Vert_{{\rm HS}(L^2,H^{\tilde \sigma})}^2
  \end{align}
with
  \begin{align}
    &\left\Vert\int_0^\cdot U_\lambda (\cdot-s)h(s)\partial_x g(s) \,ds\right\Vert_{X_1 (T)}
      \leq \CC(T)T^{1/2}
       \Vert h\Vert_{X_1(T)}
       \Vert g\Vert_{X_1(T)}
  \end{align}
for all $h,g\in X_1(T)$  and
  \begin{align}
    \Vert U_\lambda(\cdot) v_0\Vert_{X_1(T)}\leq \CC(T) \Vert v_0\Vert_{H^1}
  \end{align}
for $v_0\in H^{1}({\mathbb R})$.
Since $\CC$ is non-decreasing, we may increase it so that it is
also continuous.
For all $s\leq t$ the solutions $u$ and $v$ verify
  \begin{align}
   u(s)&={\cal T}_{u_0} (u)(s)
  \end{align}
and
  \begin{align}
   v(s)&={\cal T}_{v_0} (v)(s)
  \end{align}
where 
  \begin{equation}
    {\cal T}_{h} (g)(s)
      =U_\lambda(s)h
           - \int_0^s U_\lambda(s-r) g(r)\partial_x g(r)\,dr 
	+ \int_0^s U_\lambda(s-r) f \,dr 
           + \overline u(s)
   \label{EQ15}
  \end{equation}
for $h\in H^1(\dbR)$ and $g\in X_1(t)$.
Define 
  \begin{equation}
    \tau
       =\inf
            \left\{s\geq 0: 8\CC(t)s^{1/2}
              \left(\CC(t)R
                       +\Vert\overline u\Vert_{X_1(s)}
		+ \left\Vert\int_0^\cdot U_\lambda (\cdot-r)f \,dr\right\Vert_{X_1(s)} 
              \right)>1
            \right\}
   \period
   \label{EQ16}
  \end{equation}
By Lemma~\ref{continuity-ubar}, 
$\tau$ is a stopping time. Note  that 
  \begin{align}
 \dbP(\tau< s)&\leq \dbP\left( 8\CC(t)s^{1/2}\left(\CC(t)R+\Vert\overline u\Vert_{X_1(s)}+ \left\Vert\int_0^\cdot U_\lambda (\cdot-r)f \,dr\right\Vert_{X_1(s)} \right)>1\right)
   \nonumber\\&
  \leq \dbE\left[ 8\CC(t)s^{1/2}\left(\CC(t)R+\Vert\overline u\Vert_{X_1(s)}+ \left\Vert\int_0^\cdot U_\lambda (\cdot-r)f \,dr\right\Vert_{X_1(s)} \right)\right]
   \nonumber\\&
    \leq  \left(8\CC^2(t)s^{1/2}(R+1)\right)%\wedge 1
 \nonumber\\&
  \leq C(R)s^{1/2}%\wedge 1
  \end{align}
and thus
  \begin{align}
    \dbE[\tau]
      &=\int_0^{\infty} \dbP(s\leq \tau) \,ds =\int_0^{\infty}\bigl(1-  \dbP(s> \tau) \bigr)\,ds 
   \nonumber\\ &
   \geq  
       \int_0^{\fractext{1}{C(R)^2}}\bigl( 1-  1\wedge (C(R)s^{1/2})\bigr)\,ds 
   \nonumber\\&
   \geq \frac{1}{3C(R)^2}
   \label{EQ17}
  \end{align}
where the dependence on $t$ is understood.

{\it Step 3}:  We now inductively define a sequence of stopping times.
We start with
  \begin{align}\label{tau-def}
 &\tau_0=\tau
  \end{align}
and then for $k=0,1,\ldots$ let
  \begin{align}
    \tau_{k+1}=\inf
        \left\{s\geq \tau_k: 8\CC(t)(s-\tau_k)^{1/2} 
         \left (\CC(t)R+\Vert\overline u_s^{\tau_k}\Vert_{X_1(\tau_k,s)}+ \left\Vert\int_{\tau_k}^\cdot U_\lambda (\cdot-r)f \,dr\right\Vert_{X_1(\tau_k,s)} \right)>1 \right\}
  \end{align}
where
  \begin{align}
   \overline u^{\tau_k}_s=\int_{\tau_k}^s U_\lambda(s-r)\Phi \,dW_r
  \end{align}
and $X_1(\tau_k,s)$ is defined similarly to $X_1(T)$ for shifted process(defined on $[\tau_k,s]$). For simplicity of notation, set $\tau_{-1}=0$
 
Note that $\overline u^{\tau_k}_s=0$ for $s\leq \tau_k$.
For $s\geq \tau_k$, we have that
$\overline u^{\tau_k}_s$ is $\sigma (W_r-W_{\tau_k},r\in[\tau_k,s])$-measurable. 
Therefore, 
$\tau_{k+1}-\tau_k$ is independent from ${\mathbb G}_{\tau_k}$ and has the same distribution as $\tau$. 

By the law of large numbers, a.s.
  \begin{equation}\frac{\tau_n}{n}=\frac{1}{n} 
      \sum_{i=0}^n (\tau_i-\tau_{i-1}) \to \dbE[\tau]\geq  \frac{1}{3C^2(R)}\end{equation}
as $n\to\infty$,
where the constant is the same as in \eqref{EQ17}. 
Thus the sequence of random variables 
\begin{equation}{1}_{\{ \tau_n\leq t\}}= {1}_{\{\fractext{\tau_n}{n}- \dbE[\tau]\leq \fractext{t}{n}-\dbE[\tau]\}}\end{equation} 
converges $\dbP$-a.s. to $0$ as $n\to\infty$.
By the dominated convergence theorem $\dbP( \tau_n\leq t)\to 0$ as $n\to\infty$. Hence, there exists $n>0$ depending only on 
$(R,\epsilon, M)$ such that  
  \begin{equation}
    \dbP(\tau_n\leq t)\leq \frac{\epsilon}{6M}
   \period
   \label{EQ18}
  \end{equation}
Therefore, for all $v_0$ satisfying $\Vert u_0-v_0\Vert_{H^1}\leq 1$,
we have
  \begin{align}
   & \dbE\left[|\xi(u(t))-\xi(v(t))|\right] 
    \nonumber\\&\indeq
\leq \dbE\left[[|\xi(u_t)-\xi(v_t)|1_{\max\{\sup_{s\in[0,t]} \Vert u(s)\Vert^2_{H^1} ,\sup_{s\in[0,t]} \Vert v(s)\Vert^2_{H^1} \}\geq R }\right]
   \nonumber\\&\indeq\indeq
    +\dbE\left[[|\xi(u_t)-\xi(v_t)|1_{\{\max\{\sup_{s\in[0,t]} \Vert u(s)\Vert^2_{H^1} ,\sup_{s\in[0,t]} \Vert v(s)\Vert^2_{H^1} \}\leq R\} }1_{\{\tau_n\leq t\}}\right]
   \nonumber\\&\indeq\indeq
     +\dbE\left[[|\xi(u_t)-\xi(v_t)|1_{\{\max\{\sup_{s\in[0,t]} \Vert u(s)\Vert^2_{H^1} ,\sup_{s\in[0,t]} \Vert v(s)\Vert^2_{H^1} \}\leq R\}} 1_{\{\tau_n\geq t\}}\right]
  \nonumber\\&\indeq
 = T_1+T_2
   +\dbE\left[[|\xi(u_t)-\xi(v_t)|1_{\{\max\{\sup_{s\in[0,t]} \Vert u(s)\Vert^2_{H^1} ,\sup_{s\in[0,t]} \Vert v(s)\Vert^2_{H^1} \}\leq R\}} 1_{\{\tau_n\geq t\}}\right]
  \end{align}
Note that by the choice of $R$, we have
$T_1\leq \fractext{\epsilon}{3}$. 
Similarly, by \eqref{EQ18}, $T_2\leq\fractext{\epsilon}{3}$. Thus for all $v_0$ such that $\Vert u_0-v_0\Vert_{H^1}\leq 1$, we have
  \begin{align}\label{control-feller}
   &\dbE\left[|\xi(u(t))-\xi(v(t))|\right] 
   \nonumber\\&\indeq
   \leq \frac{2\epsilon}{3}+\dbE\left[[|\xi(u_t)-\xi(v_t)|1_{\{\max\{\sup_{s\in[0,t]} \Vert u(s)\Vert^2_{H^1} ,\sup_{s\in[0,t]} \Vert v(s)\Vert^2_{H^1} \}\leq R\}} 1_{\{\tau_n\geq t\}}\right]
   \period
  \end{align}
%\end{proof}

\nnewpage
In order to continue our analysis, we need the following lemma.

\coe
\begin{Lemma}\label{small-tau}
Denote the event
  \begin{equation}
    A=\left\{
          \max\left\{\sup_{s\in[0,t]} \Vert u(s)\Vert^2_{H^1} ,\sup_{s\in[0,t]}
      \Vert v(s)\Vert^2_{H^1} 
              \right\}\leq R
      \right\}
   \period
   \label{EQ13}
  \end{equation}
Then for all $k\in{\mathbb N}_0$ we have
the inequality 
  \begin{align}
   {\color{blue}\sup_{s\in[0,t]}} \Vert u(t)-v(t)\Vert_{H^1}\leq (2\tilde C(t))^{k+1}\Vert u_0-v_0\Vert_{H^1}
  \end{align}
on the event $A\cap \{\tau_k\geq t\}$.
\end{Lemma}
\cob

\begin{proof}[Proof of Lemma~\ref{small-tau}]
We start the induction with $k=0$.
Fix $s\in [0,t]$.
Then on the set $ \{\tau_0> s\}$, we have
  \begin{equation}
      \sup_{r\in[0,s]}\Vert u(r)-v(r)\Vert_{H^1}\leq 2\CC(t) \Vert u_0-v_0\Vert_{H^1}
   \period
  \end{equation}
Indeed, let 
  \begin{equation}
    L_0=\CC(t)R+\Vert\overline u\Vert_{X_1(s)}+ \left\Vert\int_0^\cdot U_\lambda (\cdot-r)f \,dr\right\Vert_{X_1(s)} 
   \label{EQ19}
  \end{equation}
 and let  $g\in X_1(s)$ be such that 
$\Vert g\Vert_{X_1(s)}\leq 2L_0$. Then on the set $\{\tau_0>s\}$
  \begin{align}
 \Vert{\cal T}_{u_0}(g)\Vert_{X_1(s)}&\leq\CC(s) \Vert u_0\Vert_{H^1}+ \left\Vert\int_0^\cdot U_\lambda (\cdot-r)g(r)\partial_x g(r) \,dr\right\Vert_{X_1(s)}+ \left\Vert\int_0^\cdot U_\lambda (\cdot-r)f \,dr\right\Vert_{X_1(s)} +\Vert\overline u\Vert_{X_1(s)}
\nonumber\\&
  \leq  \CC(t) R +\Vert\overline u\Vert_{X_1(s)}+ \left\Vert\int_0^\cdot U_\lambda (\cdot-r)f \,dr\right\Vert_{X_1(s)} + \CC(t)s^{1/2} \Vert g\Vert_{X_1(s)}^2\leq L_0 +4 L_0^2 \CC(t)s^{1/2}
  \nonumber\\&
  \leq L_0 +L_0/2\leq 2L_0
  \end{align}
where the last inequality holds 
due to inclusion $\{\tau_0>s\}\subset \{L_0\leq \fractext{1}{8\CC(t)s^{1/2}}\}$.

Note that $u$ is a fixed point of ${\cal T}_{u_0}$ and 
that ${\cal T}_{u_0}$ maps the ball of radius $2L_0$ of $X_1(s)$ into itself. Then $\Vert u\Vert_{X_1(s)}\leq 2L_0$. Similarly,  $\Vert v\Vert_{X_1(s)}\leq 2L_0$.
Now, observe that
$u$ (resp.~$v$) 
is a fixed point of ${\cal T}_{u_0}$ (resp.~${\cal T}_{v_0}$). Therefore,  on the set $ A\cap \{\tau_0> s\}$,
  \begin{align}
 \Vert u-v\Vert_{X_1(s)}&= \Vert{\cal T}_{u_0}(u)-{\cal T}_{v_0}(v)\Vert_{X_1(s)}\nonumber\\
  &= \Vert U_\lambda(\cdot)(u_0-v_0)\Vert_{X_1(s)}\nonumber\\
  &\indeq
+\left\Vert\int_0^\cdot U_\lambda (\cdot-r)\left((u(r)-v(r))\partial_x u(r)+v(r)\partial_x (u(r)-v(r))\right) \,dr\right\Vert_{X_1(s)}\nonumber\\
  &\leq \CC(t) \Vert u_0-v_0\Vert_{H^1}+\CC(t)s^{1/2}\Vert u-v\Vert_{X_1(s)}(\Vert u\Vert_{X_1(s)}+\Vert v\Vert_{X_1(s)})\nonumber\\
  &\leq \CC(t) \Vert u_0-v_0\Vert_{H^1}+4 \CC(t)s^{1/2} L_0 \Vert u-v\Vert_{X_1(s)} \nonumber\\
  &\leq \CC(t) \Vert u_0-v_0\Vert_{H^1}+\frac{1}{2}\Vert u-v\Vert_{X_1(s)}
  \end{align}
which implies that  on the set $ \{\tau_0> s\}$,
\begin{equation}\sup_{r\in[0,s]}\Vert u(r)-v(r)\Vert_{H^1}\leq \Vert u-v\Vert_{X_1(s)}\leq  2\CC(t) \Vert u_0-v_0\Vert_{H^1}
   \period
\end{equation}
By the continuity in time of the processes, we have
\begin{equation}\sup_{r\in[0,\tau_0]}\Vert u(r)-v(r)\Vert_{H_1}\leq  2\CC(t) \Vert u_0-v_0\Vert_{H^1}
   \period\end{equation}

We finish the proof by induction. Assume that for some $k\leq n-1$, 
we have on
the set $A\cap \{\tau_k\geq t\}$ 
  \begin{equation}\sup_{r\in[0,\tau_k]}\Vert u(r)-v(r)\Vert_{H_1}\leq  (2\CC(t))^{k+1} \Vert u_0-v_0\Vert_{H^1}
   \period
\end{equation}
In order to obtain the heredity, we need to give the upper bound on $A\cap \{\tau_{k+1}\geq t>\tau_k\}$. 

Note that by the strong Markov property for all $s\geq \tau_k$ 
  \begin{align}
   u_s&= U_\lambda (s-\tau_k)u_{\tau_k}-\int_{\tau_k}^s U_\lambda (s-r) u_r \partial_x u_r \,dr+ \int_{\tau_k}^s U_\lambda (s-r) f \,dr +\overline u^{\tau_k}_s
  \end{align}
and
  \begin{align}
   v_s&= U_\lambda (s-\tau_k)v_{\tau_k}-\int_{\tau_k}^s U_\lambda (s-r) v_r \partial_x v_r \,dr+ \int_{\tau_k}^s U_\lambda (s-r) f \,dr +\overline u^{\tau_k}_s
   \period
  \end{align}
On the set  $A\cap \{\tau_{k+1}\geq t>\tau_k\}$, we have $\Vert u_{\tau_k}\Vert_{H^1}\leq R$ and $\Vert v_{\tau_k}\Vert_{H^1}\leq R$.
Hence, we may define $L_{k+1}= \CC(t)R+\Vert\overline u^{\tau_k}\Vert_{X_1(\tau_k,s)}+ \left\Vert\int_{\tau_k}^\cdot U_\lambda (\cdot-r)f \,dr\right\Vert_{X_1(\tau_k,s)}$ for $s\in[\tau_k,t]$.  

Similarly, we can prove that on the set $A\cap \{s<\tau_{k+1}\}$ we have  $\Vert u\Vert_{X_1(\tau_k,s)}\leq 2L_{k+1}$ and $\Vert u\Vert_{X_1(\tau_k,s)}\leq 2L_{k+1}$ and proceed as in the proof for $k=0$ that 
  \begin{equation}\sup_{r\in[\tau_k,\tau_{k+1}]}\Vert u(r)-v(r)\Vert_{H_1}\leq 2\CC(t)\Vert u(\tau_k)-v(\tau_k)\Vert_{H^1}\leq (2\CC(t))^{k+2}\Vert u_0-v_0\Vert_{L^2}
  \end{equation}
and the lemma is established.
\end{proof}

We now continue the proof of the Feller property 
starting from the
inequality \eqref{control-feller}.  Using the previous lemma 
we have 
  \begin{equation}
    \Vert u(t)-v(t)\Vert_{H^1}\leq (2\CC(t))^{k+1}\Vert u_0-v_0\Vert_{H^1}
   \hbox{~~on~}A\cap \{\tau_n\geq t\}
   \period
  \end{equation}
Therefore, as $v_0\to u_0$ we have 
  \begin{align}
   |\xi(u_t)-\xi(v_t)|1_{\{\max\{\sup_{s\in[0,t]} \Vert u(s)\Vert^2_{H^1} ,\sup_{s\in[0,t]} \Vert v(s)\Vert^2_{H^1} \}\leq R\}} 1_{\{\tau_n\geq t\}}\to 0\mbox { a.s.}
  \end{align}
Using the dominated convergence theorem  
\begin{equation}\dbE\left[[|\xi(u_t)-\xi(v_t)|1_{\{\max\{\sup_{s\in[0,t]} \Vert u(s)\Vert^2_{H^1} ,\sup_{s\in[0,t]} \Vert v(s)\Vert^2_{H^1} \}\leq R\}} 1_{\{\tau_n\geq t\}}\right]\to 0\end{equation}
as $v_0\to u_0$ in $H^1$. 
Note that the choice of $R$ and $n$ does not depend on $v_0$. 

\nnewpage
\subsection{Asymptotic compactness of the semi-group}
\label{secasy}
We use the distributional convergence over various Sobolev spaces. 
In order to fix the ideas, we first recall the definition.

\begin{definition}\label{defn-convergence}
Let $\Gamma$ be a topological vector space, 
and let $\{X_n\}_{n\geq 0}$ and $X_\infty$ be random variables taking
values in $\Gamma$, possibly defined in different probability spaces.
We say that $X_n$ converges to $X_\infty$ in distribution in $\Gamma$ if
  \begin{equation}
    \dbE[F(X_n)]\to \dbE[F(X_\infty)]
    \hbox{~~as~}n\to\infty
   \label{EQ23}
  \end{equation}
%as $n\to \infty$
for all continuous bounded functions $F\colon\Gamma\to\dbR$.
\end{definition}

We shall exploit the classical results on the asymptotic compactness
of the solution operator of the KdV equation in order
to prove the following lemma.

\coe
\begin{Lemma}\label{tightness-p}
For {\color{blue} any sequence of deterministic initial condition $u_0^n$ satisfying $$R:=\sup_n \left\{\Vert u^n_0\Vert^2_{H^1} \right\}<\infty$$}
and a sequence of nonnegative numbers 
$t_1, t_2, \ldots$ 
such that $\lim_{n\to \infty} t_n=\infty$, the set of probabilities 
$\{P_{t_n}(u_0^n,\cdot):n\in \dbN\}$ is tight in $H^1$. 
\end{Lemma}
\cob

\begin{proof}[Proof of Lemma~\ref{tightness-p}]
Without loss of generality,
we may assume that $t_1, t_2, \ldots$ is increasing. 
Let $\{u_0^n\}_{n=1}^{\infty}$ be a sequence of initial conditions as above. 
We denote by $\{u^n(t)\}_{n=1}^{\infty}$ the respective solutions of \eqref{kdv-intro}. We intend to show that there is a subsequence of $\{u^n_{t_n}\}$ that converges in distribution in $H^1$. 

By Lemma~\ref{lemma-invariant-I}, we have the bound
\begin{equation}\label{bound-H1}
\sup_{n}\dbE[\Vert u^n({t_n})\Vert^2_{H^1}]\le C(R)
   \period
\end{equation}

{\it Step 1: convergence in distribution in $L^{2}_{\rm loc}(\dbR)$}

Bounded sets of $H^1(\dbR)$ are relatively compact in $L^{2}_{\rm loc}(\dbR)$. Thus the inequality \eqref{bound-H1}  and Prokhorov's theorem in $L^{2}_{\rm loc}(\dbR)$ allow us to conclude that there exists an $L^{2}_{\rm loc}$ valued random variable $\xi$ (possibly defined on another probability space) and a subsequence of $\{u^n_{t_n}\}$ 
%(still indexed with $n$) 
such that 
%as $n$ goes to infinity
  \begin{equation}\label{limit-H-1}
    u^n_{t_n}\to\xi\mbox{ in distribution in  } L^{2}_{\rm loc}
     \hbox{~as~}n\to\infty
   \period
  \end{equation}
Let $\{f_i\}_{i=1}^{\infty}$ be an orthonormal basis of $H^1(\dbR)$ 
with $f_i$ smooth and compactly supported.
For all $i\in{\mathbb N}$ 
and $M>0$ we define the  mapping $$v\to \psi_{i,M}(v):=|(v,f_i)_{H^1}|\wedge M=|(v,(1-\partial_x^2)f_i)|\wedge M$$
which is continuous in $L^{2}_{\rm loc}(\dbR)$ and bounded. 
Therefore the  distributional $L^{2}_{\rm loc}(\dbR)$ convergence imply 
$$\dbE\left[\sum_{i=1}^N \left((u^n_{t_n},f_i)^2_{H^1}\wedge M^2\right)\right]\to \dbE\left[\sum_{i=1}^N \left((\xi,f_i)^2_{H^1}\wedge M^2\right)\right].$$
% bounded, we truncate
%it and estimate the remainders with the estimate \eqref{bound-H1}.
Therefore, for all $N\in{\mathbb N}$ and $M>0$ 
  \begin{equation}
   \dbE\left[\sum_{i=1}^N \left((\xi,f_i)^2_{H^1}\wedge M^2\right)\right]\leq C(R)
   \period
  \end{equation}
  Sending $M$ to infinity, by Fatou's Lemma we obtain that 
    \begin{equation}
   \dbE\left[\sum_{i=1}^N (\xi,f_i)^2_{H^1}\right]\leq C(R)
   \label{EQ34}
  \end{equation}
and similarly sending $N$ to infinity, we get 
 \begin{equation}
  \dbE[\Vert\xi\Vert^2_{H^1}]
       \leq\liminf_N \dbE\left[\sum_{i=1}^N (\xi,f_i)^2_{H^1}\right]\leq C(R)
   \period
   \label{EQ24}
  \end{equation}
This shows that $\xi$ is $H^1(\dbR)$-valued.

{\it Step 2: convergence in distribution in $L^2$}.
In order to prove this, we use 
\coe
  \begin{equation}
   \lim_n \dbE[\Vert  u^n_{t_n}\Vert^2_{L^2}]
     =\dbE[\Vert\xi\Vert^2_{L^2}]
   \period
   \label{EQ25}
  \end{equation}
\cob
The proof of this fact
is given in the appendix.

%The proof requires the introduction of weak solutions of SPDEs which is only used to prove 
%Lemmas~\ref{convergence-l2-norms} and~\ref{convergence-I-norms}. 

%\coe
%\begin{Lemma}\label{convergence-l2-norms}
%\end{Lemma}
%\cob
%We now continue with the proof of convergence in distribution in $L^2$.
Recall from Section~\ref{sec2} 
that $\{e_i\}$ 
denotes an orthonormal basis of $L^2({\mathbb R})$
consisting
of smooth compactly supported functions.
%For all $N$ the convergence in distribution in $L^{2}_{\rm loc}(\dbR)$ and the fact that $e_i$ is smooth and compactly supported allow us to conclude that the 
For all $N\in{\mathbb N}$ and $M>0$, the convergence in distribution in $L^2_{\rm loc}(\dbR)$ of $u^n_{t_n}$ implies that   \begin{equation}
   \dbE\left[\sum_{i=1}^N (u^n_{t_n},e_i)^2\wedge M^2\right]\to\dbE\left[\sum_{i=1}^N (\xi,e_i)^2\wedge M^2\right].
  \end{equation}
Using the inequality \eqref{EQ54}, we obtain that the family $\Vert u^n_{t_n}\Vert_{L^2}^2$ is uniformly integrable. Thus we can take $M$ to infinity and obtain 
  \begin{equation}
   \dbE\left[\sum_{i=1}^N (u^n_{t_n},e_i)^2\right]\to\dbE\left[\sum_{i=1}^N (\xi,e_i)^2\right].
  \end{equation}
Therefore, combined with
%the Lemma~\ref{convergence-l2-norms}
%, for all $N\geq 0$ we have 
\eqref{EQ25}, we get
  \begin{equation}\label{convergence-remainder}
   \dbE\left[\sum_{i=N+1}^\infty (u^n_{t_n},e_i)^2\right]\to\dbE\left[\sum_{i=N+1}^\infty (\xi,e_i)^2\right]
   \period
  \end{equation}
Now, fix $\epsilon >0$. There exists $N_0\in{\mathbb N}_0$ such that 
  \begin{equation}
   \dbE\left[\sum_{i=N_0+1}^\infty (\xi,e_i)^2\right]\leq \epsilon/2   
   \period
   \label{EQ26}
  \end{equation}
Then, using \eqref{convergence-remainder}, there exists
$n_\epsilon\in{\mathbb N}$ 
such that 
  \begin{equation}
   \sup_{n\geq n_\epsilon}
       \dbE\left[\sum_{i=N_0+1}^\infty (u^n_{t_n},e_i)^2\right]\leq \epsilon   
   \label{EQ27}
  \end{equation}
By the uniform second moment bounds \eqref{bound-H1}, 
%For all $n\leq n_\epsilon-1$, 
  \begin{equation}
   \lim_{N\to \infty}
      \dbE\left[\sum_{i=N}^\infty (u^n_{t_n},e_i)^2\right]=0   
   \comma n\leq n_{\epsilon}-1
   \period
   \label{EQ29}
  \end{equation}
By \eqref{EQ27} and \eqref{EQ29}, there exists
$N_1\geq N_0$ such that
  \begin{equation}
    \sup_{n\in{\mathbb N}}
    \dbE\left[\sum_{i=N_1+1}^\infty (u^n_{t_n},e_i)^2\right]\leq \epsilon
   \period
   \label{EQ33}
  \end{equation}
Thus we have proven that 
  \begin{equation}\label{limit-Pro}
   \lim_{N\to \infty }\sup_n \dbE\left[\sum_{i=N}^\infty (u^n_{t_n},e_i)^2\right]=0
   \period
  \end{equation}
By  \cite[Theorem 1.13]{P} this implies tightness in distribution  in $L^2$ of measures of $\{u^n_{t_n}\}$. Note that any limiting measure can only be the measure of $\xi$. Thus 
\begin{equation}\label{limit-l2}
u^n_{t_n}\to \xi
\end{equation}
 in distribution in $L^2$. 

%\begin{Remark}\label{rem-convergence}
%{\rm
We emphasize that we have \emph{not} taken any further subsequence to pass from \eqref{limit-H-1} to \eqref{limit-l2}. We have proven that any limit in distribution in $L^{2}_{\rm loc}(\dbR)$ of $\{u^n_{t_n}\}$ is also its limit in distribution in $L^2(\dbR)$. 
%}
%\end{Remark}

{\it Step 3: Convergence in distribution in $H^1$}

The fundamental tool is the fact
\coe
%\begin{Lemma}\label{convergence-I-norms}
  \begin{equation}
   \dbE[I(\xi)]=\lim_n \dbE[I(u^n_{t_n})]
   \label{EQ35}
  \end{equation}
%\end{Lemma}
\cob
the proof of which is given in the appendix.
%We continue with the proof of tightness in distribution in $H^1$.
Note that we have the uniform bounds \eqref{bound-H1} and convergence
in distribution in $L^2$. Using Agmon's inequality, we obtain that the mapping $v\to \int v^3(x)dx$ is continuous in $L^2$ on bounded sets of $H^1$. 
Thus, using again the uniform integrability of  the families $\Vert \partial_x u\Vert_{L^2}^2$ and $\Vert  u\Vert_{L^2}^2$,
   \begin{equation}\dbE\left[\int (u^n_{t_n})^3(x)dx\right]\to\dbE\left[\int (\xi)^3(x)dx\right]\period\label{EQ45}\end{equation} 

Combined with 
%Lemma~\ref{convergence-I-norms} 
\eqref{EQ35}
this implies
\begin{equation}\lim_n \dbE[\Vert\partial_x u^n_{t_n}\Vert_{L^2}^2]=\dbE[\Vert\partial\xi\Vert_{L^2}^2]
   \period\label{EQ44}\end{equation}
Note that the inequality \eqref{EQ09}, for $k=2$ gives the uniform integrability of $\Vert u^n_{t_n}\Vert^2_{H^1}$.
Repeating for the space $H^1$ the same ideas that allowed us to obtain the convergence in distribution in $L^2$ we obtain
  \begin{equation}u^n_{t_n}\to \xi\label{EQ46}\end{equation} 
in distribution in $H^1$. 
\end{proof}
Before proving the lemma \ref{lemma-tightness} we prove the following lemma.

\coe
\begin{Lemma}\label{tightness-on-compacts}
Let $K$ be a compact subset of $H^1(\dbR)$. Then the set of measures on $H^1(\dbR)$
$$\{P_{s}(v,\cdot):s\in[0,1],\, v\in K\}$$
is tight. 
\end{Lemma}
\cob
\begin{proof}
We will take a countable subset $\{P_{s}(v,\cdot):s\in[0,1],\, v\in K\}$ and show that it has a convergent subsequence. 
Let $(s^n,v^n)\in [0,1]\times K$. By compactness of the sets there exists a subsequence of $(s^n,v^n)$ (still denoted $(s^n,v^n)$) the converges to $(s,v)\in [0,1]\times K$.
Denote by $u^n$ the solution of \eqref{kdv-intro-A} with initial data $v^n$ and by $u$ the solution of \eqref{kdv-intro-A} with intial data $v$. 

We now prove that there exists a subsequence $(s^{n_k},v^{n_k})$ such that 
\begin{align}\label{conv-sub-seq}
\dbP\mbox{-a.s. } \lim_k \Vert u^{n_k}_s- u_s\Vert_{H^1}+\Vert u_{s^{n_k}}-u^s\Vert_{H^1}\to 0 \mbox{ as }k\to \infty.
\end{align}
The almost sure convergence $\Vert u_{s^n}-u^s\Vert_{H^1}\to 0$ is a direct consequence of $u\in C([0,1];H^1)$. Fix $\epsilon>0$ and $\delta>0$ and similarly to the proof of the Feller property denote $R_0=\sup_{v\in K}\Vert v\Vert_{H^1}+1.$ We choose $R>0$ (independent of $n$)such 
  \begin{align*}
    & \dbP\left(
            \max\left\{
                \sup_{s\in[0,t]}\Vert u(s)\Vert^2_{H^1},
                \sup_{s\in[0,t]}\Vert u^n(s)\Vert^2_{H^1}\right\}\geq R\right)
    \nonumber\\&\indeq
     \leq\frac{1}{R} \dbE\left[\sup_{s\in[0,t]}\Vert u(s)\Vert^2_{H^1}+\sup_{s\in[0,t]}\Vert u^n(s)\Vert^2_{H^1}\right]
   \nonumber\\&\indeq
 \leq\frac{C(R_0)}{R}\leq \epsilon /2.
   \period
  \end{align*}
Given the choice of $R$ we define the hitting times $\tau_k$ as in \eqref{tau-def}(with $t=1$). We choose $N$ such that $$\dbP(\tau_N\leq 1)\leq \epsilon/2.$$ 
We define the events    
\begin{align*}
    A^n=\left\{
          \max\left\{\sup_{s\in[0,t]} \Vert u(s)\Vert^2_{H^1} ,\sup_{s\in[0,t]}
      \Vert u^n(s)\Vert^2_{H^1} 
              \right\}\leq R
      \right\}
   \period
  \end{align*}
Thanks to Lemma \ref{small-tau}, on the set $A^n\cap \{\tau_N\geq 1\}$ one has 
  \begin{align*}
   \sup_{s\in[0,1]}\Vert u(t)-u^n(t)\Vert_{H^1}\leq (2\tilde C(1))^{N+1}\Vert v-v^n\Vert_{H^1}.
  \end{align*}
  We choose $n$ large enough to have $ (2\tilde C(1))^{N+1}\Vert v-v^n\Vert_{H^1}\leq \delta$.
This implies that for all $n$ large enough
$$\dbP\left(\sup_{s\in[0,1]} \Vert u(t)-u^n(t)\Vert_{H^1}\leq \delta \right) \geq \dbP\left(A^n\cap \{\tau_N\geq 1\}\right)\geq 1-\epsilon$$
which is exactly $\sup_{s\in[0,1]} \Vert u(t)-u^n(t)\Vert_{H^1}\to 0$ as $n\to \infty$ in probability. Hence we can take a subsequence that converges almost surely. 
Let $\xi$ be a real valued uniformly continuous function on $H^1(\dbR)$. 
By direct estimates 
\begin{align*}
|P_{s^{n_k}}\xi(v^{n_k})-P_{s}\xi(v)| \leq \dbE\left[|\xi (u^{n_k}_{s^{n_k}})-\xi (u_{s^{n_k}})|\right]+\dbE\left[|\xi (u_{s^{n_k}})-\xi (u_{s})|\right].
\end{align*}
The dominated convergence theorem, the convergence \eqref{conv-sub-seq} and the uniform continuity of $\xi$ imply that the right hand side goes to $0$. 
\end{proof}

\begin{proof}[Proof of Lemma~\ref{lemma-tightness}]
Fix $\epsilon >0$.  The asymptotic compactness of the equation implies that the set of probabilities on $H^1(\dbR)$
  \begin{equation}
   \{P_n(0,\cdot); n\geq 0\}
  \end{equation}
is tight. We choose a compact set $K_\epsilon\subset H^1(\dbR)$ such that 
$$\sup_n P_n(0,K_\epsilon^c)\leq \epsilon/2.$$
Additionally by the lemma \ref{tightness-on-compacts} the set of probabilities on $H^1(\dbR)$
$$\{P_s(v,\cdot); s\in[0,1],\, v\in K_\epsilon\}$$
is also tight. We pick another compact $A_\epsilon\subset H^1(\dbR)$ such that 
$$\sup_{s\in[0,1],\,v\in K_\epsilon}P_s(v,A_\epsilon^c)\leq \epsilon/2.$$

By direct computation
  \begin{align}
   \mu_n (A_\epsilon^c)
    &=\frac{1}{n}\int_0^n \dbP(u(t)\in A_\epsilon^c) \,dt
 \nonumber\\&
   =\frac{1}{n}\sum_{i=0}^{n-1} \int_i^{i+1} \int_{H^1}P_{i}(0; dy) P_{t-i}(y; A_\epsilon^c)  \,dt
    \nonumber\\&
   =\frac{1}{n}\sum_{i=0}^{n-1} \int_i^{i+1}\left\{ \int_{K_\epsilon} P_{i}(0; dy) P_{t-i}(y; A_\epsilon^c)  +  \int_{K_\epsilon^c} P_{i}(0; dy) P_{t-i}(y; A_\epsilon^c)\right\}\,dt
  \nonumber\\&
   \leq\frac{1}{n}\sum_{i=0}^{n-1} \int_i^{i+1}\left\{ \int_{K_\epsilon} P_{i}(0; dy) \epsilon/2 +   P_{i}(0; K_\epsilon^c) \right\}\,dt\leq \epsilon
   \period
   \label{EQ30}
  \end{align}
Thus 
$\mu_n (A_\epsilon^c)\leq \epsilon$,
and the proof is concluded.
\end{proof}

\nnewpage
\startnewsection{Appendix}{sec5}

%\subsection{Proof of Lemma~\ref{convergence-l2-norms}}
\subsection{Proof of $\lim_n \dbE[\Vert u^n_{t_n}\Vert^2_{L^2}]\to\dbE[\Vert\xi\Vert^2_{L^2}]$}
Note that the inequality 
$\liminf_n \dbE[\Vert u^n_{t_n}\Vert^2_{L^2}]\geq\dbE[\Vert\xi\Vert^2_{L^2}]$ 
can be
shown easily.
In order to prove the reverse inequality,
assume, contrary to the assertion,
that there exists $\epsilon>0$ and a subsequence of $\{u^n_{t_n}\}$ (still indexed by $n$) such that for all $n\geq 0$
  \begin{equation}
  \dbE[\Vert u^n_{t_n}\Vert^2_{L^2}]-\dbE[\Vert\xi\Vert^2_{L^2}]\geq
     \epsilon
   \period
   \label{EQ37}
  \end{equation}
Fix $T>0$ such that $3 C(R) e^{-2 \lambda T}\leq \epsilon$ where $C(R)$
is the constant in \eqref{bound-H1}. Note that the sequence
$\{u^n_{t_n-T}\}$ satisfies the same assumptions as $\{u^n_{t_n}\}$
and thus there exists a further subsequence (still indexed by $n$) and
$\xi_{-T}$, an $H^1$ valued random variable, such that we also have
\begin{equation}u^n_{t_n-T}\to\xi_{-T}\label{EQ47}\end{equation} in distribution in $L^{2}_{\rm loc}(\dbR)$.

We shall work on the space 
$\cZ=C([0,T]; L^{2}_{\rm loc}(\dbR))$.
Denote by $z$ the canonical process on this space 
and $\cD$ its right continuous filtration. 

\begin{definition}\label{defn-weak-sol}
{\rm A measure $\nu$ on $\cZ$  is a solution of the
equation \eqref{kdv-intro} if for all $\phi$ smooth and compactly
supported functions
  \begin{equation}
      M^\phi_t=(z_t-z_0,\phi) 
         + \int_0^t (\partial_x^3 z_s+z_s\partial_x z_s +\lambda z_s-f,\phi) ds
   \label{EQ38}
  \end{equation}
and 
  \begin{equation}(M^\phi_t)^2-t\sum_i (\Phi e_i,\phi)^2\label{EQ48}\end{equation}
are $\nu$ local-martingales.}

\end{definition}
\cob

Define the sequence of measures 
  \begin{equation}
      \nu^n(dz) 
        =\int_\Omega \delta_{\{\{u^n_{t_n-T+s}(\omega)\}_{s\in [0,T]}\}}(dz)\dbP(d\omega)
   \label{EQ39}
  \end{equation}
 on $\cZ$. We shall prove by the Aldous criterion (\cite[Theorem~16.10]{B}) 
that the sequence 
$\{\nu^n\}_{n=1}^{\infty}$ is tight in distribution in $\cZ$. The first step is the following estimate.

\coe
\begin{Lemma}\label{compactness-of-measures}
We have
$\dbE[\Vert u^n_{T_n+d_n}-u^n_{T_n}\Vert_{L^{2}}^2]\to 0$ for all stopping times $T_n$ and for all $d_n\to 0$.
\end{Lemma}
\cob
\begin{proof}[Proof of Lemma~\ref{compactness-of-measures}]
Denote 
$A=(1-\partial_x^2)$ and $U^n_s=A^{-1}(u^n_{s}-u^n_{T_n})$. Applying Ito's lemma, we get
  \begin{align}
   \Vert u^n_{T_n+d_n}-u^n_{T_n}\Vert^2_{L^2}&= 
     \int_{T_n}^{T_n+d_n}
      \left(-2(\partial_x^3 u^n_s+ u^n_s\partial u^n_s+\lambda u^n_s-f,u^n_s-u^n_{T_n})+\Vert\Phi\Vert_{{\rm HS}(L^2,L^2)}^2\right)\,ds
   \nonumber\\&\indeq
   + 2 \int_{T_n}^{T_n+d_n}(\Phi dW_s,u^n_s-u^n_{T_n})
   \nonumber\\ &=
    \int_{T_n}^{T_n+d_n}
        \left(-2(\partial_x^3 u^n_s+ u^n_s\partial u^n_s+\lambda u^n_s-f,u^n_s)+\Vert\Phi\Vert_{{\rm HS}(L^2,L^2)}^2\right)\,ds
   \nonumber\\&\indeq
   +2\left(u^n_{T_n+d_n}-u^n_{T_n}-\int_{T_n}^{T_n+d_n}\Phi dW_s ,u^n_{T_n}\right) 
   \nonumber\\&\indeq
   + 2 \int_{T_n}^{T_n+d_n}(\Phi dW_s,u^n_s-u^n_{T_n})
   \label{EQ36}
  \end{align}
Now, we proceed to bound the terms on the far right side of the above equality.
We first use $u^n_s\in H^3(\dbR)$ and $(\partial_x^3 u^n_s+ u^n_s\partial_x u^n_s,u^n_s)=0$ to get
  \begin{align}
   |(\partial_x^3 u^n_s+ u^n_s\partial_x u^n_s+\lambda u^n_s-f,u^n_s)|=|( \lambda u^n_s-f,u^n_s)|\leq C(   \Vert u_s^n\Vert_{L^2}^2 +1)
  \end{align}
which is bounded in expectation. 
The difficult term is
  \begin{align}
   &\left|(u^n_{T_n+d_n}-u^n_{T_n}-\int_{T_n}^{T_n+d_n}\Phi dW_s ,u^n_{T_n}) \right|
   \nonumber\\&\indeq
     \leq \left(\left\Vert u^n_{T_n+d_n}-u^n_{T_n}\right\Vert_{H^{-1}} 
      +\left\Vert\int_{T_n}^{T_n+d_n}\Phi dW_s\right\Vert_{H^{-1}}\right) 
    \Vert u^n_{T_n}\Vert_{H^1}
   \period
   \label{EQ52}
  \end{align}
This shows that in order to obtain an estimate
in $L^2$ one only needs the estimate in $H^{-1}$. 
Note that the bound on 
$\dbE[\Vert\int_{T_n}^{T_n+d_n}\Phi dW_s\Vert_{H^{-1}}]$ can be easily
obtained by the Burkholder-Davis-Gundy inequality. In order to 
control $\Vert u^n_{T_n+d_n}-u^n_{T_n}\Vert_{H^{-1}}$, we apply the Ito's lemma 
and obtain
  \begin{align}
   \Vert u^n_{T_n+d_n}-u^n_{T_n}\Vert^2_{H^{-1}}
          &= \int_{T_n}^{T_n+d_n}
               \Bigl(
               -2(\partial_x^3 u^n_s+ u^n_s\partial u^n_s+\lambda
	       u^n_s-f,U^n_s)+\Vert A^{-1/2}\Phi\Vert_{{\rm HS}(L^2,L^2)}^2
              \Bigr)\,ds 
       \nonumber\\&\indeq
          + 2 \int_{T_n}^{T_n+d_n}(\Phi dW_s,U^n_s)
   \period
   \label{EQ40}
  \end{align}
First note that 
$\dbE\left[|( u^n_s\partial u^n_s+\lambda u^n_s-f,U^n_s)|\right]$ is bounded due to uniform $L^2(\dbR)$ and $H^1(\dbR)$ bounds in \eqref{EQ54} and \eqref{EQ09}. One may also control $ \dbE[\int_{T_n}^{T_n+d_n}|(\Phi dW_s,U^n_s)|]$ by the Burkholder-Davis-Gundy-inequality. The main term is 
  \begin{equation}
   \dbE\left[\int_{T_n}^{T_n+d_n}(\partial_x^3 u^n_s,U^n_s)ds\right]   
   \label{EQ41}
  \end{equation}
which we 
estimate as
  \begin{align}
   \dbE\left[|(\partial_x^3 u^n_s,U^n_s)|\right]&=\dbE\left[|(A^{-1}\partial_x^3 u^n_s,u^n_{s}-u^n_{T_n+d_n})|\right]\leq \dbE\left[\Vert A^{-1} \partial_x^3 u^n_s\Vert_{L^2}^2\right]+\dbE\left[\Vert u^n_{s}-u^n_{T_n}\Vert_{L^2}^2\right]
    \nonumber\\&
     \leq \dbE\left[\Vert u^n_s\Vert_{H^1}^2\right]+\dbE\left[\Vert u^n_{s}-u^n_{T_n}\Vert_{L^2}^2\right]\leq 4C(R)
   \period
   \label{EQ42}
  \end{align}
Finally, combining all the estimates we obtain that $\dbE[\Vert u^n_{T_n+d_n}-u^n_{T_n}\Vert_{L^2}^2]\leq C d_n\to 0$ as $n\to\infty$.
\end{proof}
\coe
\begin{Lemma}\label{limiting-measure}
The family of measures $\nu^n$ is tight over $\cZ$ and any limiting measure $\nu$ of this sequence is a solution of \eqref{kdv-intro}. Additionally, the distribution of $z_0$ (resp.~$z_T$) under $\nu$ is the same as the distribution of $\xi_{-T}$ (resp.~$\xi$).
\end{Lemma}
\cob

\begin{proof}[Proof of Lemma~\ref{limiting-measure}]
The tightness follows directly from the Lemma~\ref{compactness-of-measures} 
and the Aldous criterion, \cite[Theorem~16.10]{B}.

We first show the solution property of the limiting measure. 
Let $\phi$ be a smooth compactly supported function,
let $0\leq s_1\leq \cdots\leq s_k\leq s<t$, and
assume that $g\colon\dbR^k\to\dbR$ continuous and bounded. 
Since $u^n$ is a solution under $\nu^n$, we have
  \begin{align}
   \dbE^{\nu^n}\left[ g(M^\phi_{s_1},\ldots,M^\phi_{s_k})M^\phi_{t}\right]&=\dbE^{\nu^n} \left[ g(M^\phi_{s_1},\ldots,M^\phi_{s_k})M^\phi_{s}\right]
  \end{align}
and
  \begin{align}
\dbE^{\nu^n}\left[ g(M^\phi_{s_1},\ldots,M^\phi_{s_k})((M^\phi_{t})^2-t\sum_i (\Phi e_i,\phi)^2)\right]&=\dbE^{\nu^n} \left[ g(M^\phi_{s_1},\ldots,M^\phi_{s_k})((M^\phi_{s})^2-s\sum_i (\Phi e_i,\phi)^2)\right]
   \period
  \end{align}
The mappings that we are integrating are continuous in $z$ under the topology of $\cZ$, but they are not bounded. 
However, the bound \eqref{bound-H1} allows us to truncate them,
obtain a uniform estimate on the remainder, and pass to the limit in $n$.
We obtain
  \begin{align}
   \dbE^{\nu}\left[ g(M^\phi_{s_1},\ldots,M^\phi_{s_k})M^\phi_{t}\right]
        &=\dbE^{\nu} \left[ g(M^\phi_{s_1},\ldots,M^\phi_{s_k})M^\phi_{s}\right]
  \end{align}
and
  \begin{align}
    \dbE^{\nu}\left[ g(M^\phi_{s_1},\ldots,M^\phi_{s_k})((M^\phi_{t})^2-t\sum_i (\Phi e_i,\phi)^2)\right]&=\dbE^{\nu} \left[ g(M^\phi_{s_1},\ldots,M^\phi_{s_k})((M^\phi_{s})^2-s\sum_i (\Phi e_i,\phi)^2)\right]
   \period
  \end{align} 
Thus $\nu$ is a solution of \eqref{kdv-intro}.
We already had the convergence in distribution in $L^{2}_{\rm loc}(\dbR)$. Now,
  \begin{align}
   u^n_{t_n}\to \xi
  \end{align}
and
  \begin{align}
    u^n_{t_n-T}\to \xi_{-T}
   \period
  \end{align}
The distribution of $z_0$ (resp.~$z_T$) under $\nu$ 
is the distribution of $\xi_{-T}$ (resp.~$\xi$), which
completes the proof of Lemma~\ref{limiting-measure}.
\end{proof}

We now finish the proof of 
%Lemma~\ref{convergence-l2-norms}
\eqref{EQ25}. 
Under $\nu$, given the quadratic variation of 
$z_t-z_0+\int_0^t (\partial_x^3 z_s+z_s\partial_x z_s +\lambda z_s-f) ds$, there exists a sequence of $\nu$ Brownian motions $\tilde B^i$ such that $z_t-z_0+\int_0^t \partial_x^3 z_s+z_s\partial_x z_s +\lambda z_s-f ds= \sum_i \Phi e_i \tilde B^i_t$.

Under $\nu$, 
the process $z$ is $H^1(\dbR)$-valued. Let $z^k_0$ 
be a sequence in $H^3(\dbR)$ converging to $z_0$ in $H^1(\dbR)$ and let
$z^k$ be the associated solutions of \eqref{kdv-intro} in the probability space $(\cZ, \cD,\nu)$. By \cite[Lemma~3.2]{DD1}, $z^k_s\in H^3(\dbR)$ for all $s\in [0,T]$.  
Applying Ito's Lemma to $z^k$, we get
that the difference
  \begin{equation}
    \Vert z_t^k\Vert^2_{L^2}-\Vert z_0^k\Vert^2_{L^2}+2\lambda \int_0^t\Vert z_s^k\Vert^2_{L^2} ds-\int_0^t(z_s^k,f)ds-\Vert\Phi\Vert_{{\rm HS}}^2t
   \label{EQ43}
  \end{equation}
defines a martingale.
Taking the expectation under $\nu$  we find that 
    \begin{equation}\dbE^\nu[\Vert z_T^k\Vert^2_{L^2}]-e^{-2\lambda T}\dbE^\nu[\Vert z_0^k \Vert^2_{L^2}]=\int_0^T e^{-2\lambda (T-s)} \left(\dbE^\nu[(z_s^k,f)]+\Vert\Phi\Vert_{{\rm HS}}^2\right)ds\period\end{equation}
By the Feller property and the convergence of $z^k_0$, taking the limit as $k$ goes to $\infty$, we obtain that 
\begin{equation}\dbE^\nu[\Vert z_T\Vert^2_{L^2}]-e^{-2\lambda T}\dbE^\nu[\Vert z_0\Vert^2_{L^2}]=\int_0^T e^{-2\lambda (T-s)} \left(\dbE^\nu[(z_s,f)]+\Vert\Phi\Vert_{{\rm HS}}^2\right)ds
   \period\label{EQ49}\end{equation}
By an assumption the mapping $v\to (v,f)$ is continuous in $L^2_{\rm loc}(\dbR)$. Thus 
  \begin{align}
   \dbE^\nu[\Vert \xi\Vert^2_{L^2}]-e^{-2\lambda T}\dbE^\nu[\Vert \xi_{-T}\Vert^2_{L^2}]&=\dbE^\nu[\Vert z_T\Vert^2_{L^2}]-e^{-2\lambda T}\dbE^\nu[\Vert z_0\Vert^2_{L^2}]\nonumber\\
&=\int_0^T e^{-2\lambda (T-s)} \left(\dbE^\nu[(z_s,f)]+\Vert\Phi\Vert_{{\rm HS}}^2\right)ds\nonumber\\
&=\lim_n \int_0^T e^{-2\lambda (T-s)} \left(\dbE[(u^n_{t_n-T+s},f)]+\Vert\Phi\Vert_{{\rm HS}}^2\right)ds\nonumber\\
&=\lim_n \dbE[\Vert u^n_{t_n}\Vert^2_{L^2}]-e^{-2\lambda T}\dbE^\nu[\Vert u^n_{t_n-T}\Vert^2_{L^2}]
   \period
\end{align}
Note that $e^{-2\lambda T}  \dbE[\Vert\xi_{-T}\Vert^2_{L^2}]+e^{-2\lambda T}  \dbE[\Vert u^n_{t_n-T}\Vert^2_{L^2}]\leq 2e^{-2\lambda T}  C(R)$. Thus using the previous inequality and the choice of $T$
\begin{align}
2\epsilon/3\geq  2C(R) e^{-2\lambda T}  \geq \limsup_n  \dbE[\Vert u^n_{t_n}\Vert^2_{L^2}]-\dbE[\Vert\xi\Vert^2_{L^2}]\geq \epsilon
\end{align}
which is a contradiction, thus concluding the proof
of \eqref{EQ25}.
%Therefore, 
%$\dbE[\Vert u^n_{t_n}\Vert^2_{L^2}]\to\dbE[\Vert\xi\Vert^2_{L^2}]$

%\subsection{Proof of Lemma~\ref{convergence-I-norms}}
\subsection{Proof of $\dbE[I(\xi)]=\lim_n \dbE[I(u^n_{t_n})]$}
Recall that  
\begin{equation}u^n_{t_n}\to \xi\label{EQ50}\end{equation} in distribution in $L^2$.
We assume again that the convergence $\dbE[I(\xi)]=\lim_n \dbE[I(u^n_{t_n})]$ does not hold. As in the previous section, there exists a further subsequence(denoted similarly) and $\epsilon>0$ such that
\begin{equation}|\dbE[I(\xi)]- \dbE[I(u^n_{t_n})]|\geq \epsilon
   \period\label{EQ51}\end{equation}
Given the uniform estimates,
 one can prove that there exists a constant dependent on $R$ such that $\sup_t |\dbE[I(u_t)]|+|\dbE[\Vert u_t\Vert^4_{L^2}]|\leq C$. 
We fix $T$ such that $3 C e^{-2\lambda T}\leq \epsilon$.

Using the same notation and results in the first part of the appendix, we have that 
$\nu^n\to \nu$ in distribution in $\cZ$. 

Thus for all $s\in[0,T]$, we have $u^n_{t_n-T+s}\to z_s$ in distribution in 
$L^{2}_{\rm loc}(\dbR)$.  
Note that to pass from convergence \eqref{limit-H-1} to the convergence \eqref{limit-l2} 
we haven't needed to 
pass to a subsequence.
%using the Remark
%\ref{rem-convergence} 
%for $u^n_{t_n-T+s}$ and $z_s$ 
Thus one may show similarly that for all $s\in[0,T]$ 
the convergence of $u^n_{t_n-T+s}$ to $z_s$ is in fact in distribution in
$L^2(\dbR)$. Additionally, the mapping $v\to( \int v^3(x) dx,(\partial_x v,\partial_x f)-(v^2,f), \int
v(x)\sum_i |\Phi e_i|^2 dx)$ is continuous in $L^2(\dbR)$ on bounded
sets of $H^1(\dbR)$.

Therefore, as $n$ goes to infinity,
  \begin{align}
   &\int_0^Te^{-2\lambda (T-s)}\biggl(  
    \dbE\left[\frac{\lambda}{3}\int_\dbR u^n({t_n-T+s},x)^3 dx+ 2(\partial_x u^n(t_n-T+s),\partial_x f)-((u^n(t_n-T+s))^2,f)\right] 
   \nonumber\\&\indeq\indeq\indeq\indeq\indeq
    +\Vert\partial_x\Phi\Vert^2_{{\rm HS}(L^2,L^2)} 
    \nonumber
      -\sum_i \int\dbE[u^n(t_n-T+s,x)]\,|\Phi e_i|^2 \,dx\biggr)\,ds
  \end{align}
converges to
  \begin{align}
   &\int_0^Te^{-2\lambda (T-s)}\biggl( 
    \dbE^\nu\left[\frac{\lambda}{3}  \int_\dbR z({s},x)^3 dx+2(\partial_x z(s),\partial_x f)-(z^2(s),f)\right]    \nonumber\\&\indeq\indeq\indeq\indeq\indeq
 +\Vert\partial_x\Phi\Vert^2_{{\rm HS}(L^2,L^2)} 
    \nonumber
      -\sum_i \int\dbE^\nu[z(s,x)]\,|\Phi e_i|^2 \,dx\biggr)\,ds
   \period
  \end{align}
We have that
$\nu$ is a solution of \eqref{kdv-intro} and similarly to the previous section we can approximate the process $\{z_s\}$ by $\{z_s^k\}$, appeal to the Feller property in $H^1(\dbR)$, and show that the equality 
  \begin{align}
  \dbE^\nu [I(z_T))]&-e^{-2\lambda T}\dbE^\nu[I(z_0)]
  \nonumber\\&\indeq
   =\int_0^Te^{-2\lambda (T-s)}\biggl( 
    \dbE^\nu\left[ \frac{\lambda}{3} \int_\dbR z(z,x)^3 dx+2(\partial_x z(s),\partial_x f)-(z^2(s),f)\right] 
    \nonumber\\&\indeq\indeq\indeq\indeq\indeq\indeq
      \Vert\partial_x\Phi\Vert^2_{{\rm HS}(L^2,L^2)} -\sum_i \int\dbE^\nu [z(s,x)]\,|\Phi e_i|^2 \,dx\biggr)\,ds
  \end{align}
holds.
This finally gives 
  \begin{align}
   &\dbE [I(\xi))]-e^{-2\lambda T}\dbE^\nu[I(\xi_{-T})]
   \nonumber\\&\indeq
   =  \dbE^\nu [I(z_T))]-e^{-2\lambda T}\dbE^\nu[I(z_0)]
   =\lim_n  \dbE [I(u^n_{t_n}))]-e^{-2\lambda T}\dbE^\nu[I(u^n_{t_n-T})]
   \period
   \label{EQ10}
  \end{align}
We finish similarly to the previous section. Namely,
  \begin{align}
   \epsilon &\leq \liminf_n |\dbE[I(u^n_{t_n})]-\dbE[I(\xi)]|\leq e^{-2\lambda T }\liminf_n  |\dbE[I(u^n_{t_n-T})]-\dbE[I(\xi_{-T})]|
   \nonumber\\&
  \leq e^{-2\lambda T} 2C \leq 2\epsilon/3
   \label{EQ53}
  \end{align}
which gives the contradiction.

\section*{Acknowledgments} 
We thank Remigijus Mikulevicius for useful suggestions.
I.K.~was supported in part by the NSF grant DMS-1311943,
while M.Z.~was supported in part by the NSF grant
DMS-1109562.

\end{document}